\documentclass[british,12pt,x11names]{article}

\usepackage{babel}
\usepackage{graphicx}

\usepackage{multicol}
\usepackage[all]{xy}
\usepackage{enumerate,tikz}

\usepackage{amsmath,amssymb,amsthm}

\usepackage{amsmath}
\allowdisplaybreaks
\usepackage{amssymb}
\usepackage{amsthm}

\newtheorem{theorem}{Theorem}[section]
\newtheorem{lemma}[theorem]{Lemma}
\newtheorem{proposition}[theorem]{Proposition}
\newtheorem{corollary}[theorem]{Corollary}

\theoremstyle{definition}
\newtheorem{remark}[theorem]{Remark}
\newtheorem{remarks}[theorem]{Remarks}

\makeatletter
\def\@map#1#2[#3]{\mbox{$#1 \colon\thinspace #2 \longrightarrow #3$}}
\def\map#1#2{\@ifnextchar [{\@map{#1}{#2}}{\@map{#1}{#2}[#2]}}
\g@addto@macro\@floatboxreset\centering
\makeatother

\makeatletter

\renewcommand{\p@enumii}{}

\def\@enum@{\list{\csname label\@enumctr\endcsname}%
	{\usecounter{\@enumctr}\def\makelabel##1{
			\normalfont\ignorespaces\emph{{##1}~}}
		\setlength{\labelsep}{3pt}
		\setlength{\parsep}{0pt}
		\setlength{\itemsep}{0pt}
		\setlength{\leftmargin}{0pt}
		\setlength{\labelwidth}{0pt}
		\setlength{\listparindent}{\parindent}
		\setlength{\itemsep}{0pt}
		\setlength{\itemindent}{0pt}
		\topsep=3pt plus 1pt minus 1 pt}}
\makeatother

\newcommand{\torus}{\mathbb{T}^2}
\newcommand{\klein}{\mathbb{K}^2}
\newcommand{\z}{\mathbb{Z}}

\newcommand{\rtwo}{\mathbb{R}^2}

\renewcommand{\hom}{{\rm Hom}}

\newcommand{\ab}{{\text{Ab}}}
\renewcommand{\to}{\ensuremath{\longrightarrow}}
\newcommand{\zn}{\mathbb{Z}_n}

\renewcommand{\ker}[1]{\ensuremath{\operatorname{\text{Ker}}\left({#1}\right)}}
\newcommand{\im}[1]{\ensuremath{\operatorname{\text{Im}}\left({#1}\right)}}

\usepackage[
bookmarks=true,
breaklinks=true,
bookmarksnumbered=true,
colorlinks= true,
urlcolor= green,
anchorcolor=yellow,
citecolor=blue,
]{hyperref}                   

\makeatletter

\renewcommand{\p@enumii}{}
\makeatother

\begin{document}
	
	\title{Free cyclic actions on surfaces and the Borsuk-Ulam theorem}
	
	\author{DACIBERG LIMA GON\c{C}ALVES
		~\footnote{Departamento de Matem\'atica, IME, Universidade de S\~ao Paulo, Rua do Mat\~ao 1010 CEP: 05508-090, S\~ao Paulo-SP, Brazil. 
			e-mail: \texttt{dlgoncal@ime.usp.br}}
		\and
		JOHN GUASCHI
		~\footnote{Normandie Univ, UNICAEN, CNRS, LMNO, 14000 Caen, France.
			e-mail: \texttt{john.guaschi@unicaen.fr}}
		\and
		VINICIUS CASTELUBER LAASS
		~\footnote{Departamento de Matem\'atica, IME, Universidade Federal da Bahia, Av.\ Milton Santos, S/N Ondina CEP: 40170-110, Salvador-BA, Brazil. 
			e-mail: \texttt{vinicius.laass@ufba.br}} 
	}
	
	
	\maketitle
	
	\begin{abstract}%
Let $M$ and $N$ be topological spaces, let $G$ be a group, and let $\tau \colon\thinspace G \times M \to M$ be a proper free action of $G$. In this paper, we define a Borsuk-Ulam-type property for homotopy classes of maps from $M$ to $N$ with respect to the pair $(G,\tau)$ that generalises the classical antipodal Borsuk-Ulam theorem of maps from the $n$-sphere $\mathbb{S}^n$ to $\mathbb{R}^n$. In the cases where $M$ is a finite pathwise-connected CW-complex, $G$ is a finite, non-trivial Abelian group, $\tau$ is a proper free cellular action, and $N$ is either $\mathbb{R}^2$ or a compact surface without boundary different of $\mathbb{S}^2$ and $\mathbb{RP}^2$, we give an algebraic criterion involving braid groups to decide whether a free homotopy class $\beta \in [M,N]$ has the Borsuk-Ulam property. As an application of this criterion, we consider the case where $M$ is a compact surface without boundary equipped with a free action $\tau$ of the finite cyclic group $\mathbb{Z}_n$. In terms of the orientability of the orbit space $M_\tau$ of $M$ by the action $\tau$, the value of $n$ modulo $4$ and a certain algebraic condition involving the first homology group of $M_\tau$, we are able to determine if the single homotopy class of maps from $M$ to $\mathbb{R}^2$ possesses the Borsuk-Ulam property with respect to $(\mathbb{Z}_n,\tau)$. Finally, we give some examples of surfaces on which the symmetric group acts, and for these cases, we obtain some partial results regarding the Borsuk-Ulam property for maps whose target is $\mathbb{R}^2$.
	\end{abstract}
	
	\begingroup
	\renewcommand{\thefootnote}{}
	\endgroup
	

	\section{Introduction}\label{sec:introduction}
	
The famous Borsuk-Ulam theorem states that every continuous map from the $m$-dimensional sphere to the $m$-dimensional Euclidean space collapses a pair of antipodal points \cite[Satz II]{Borsuk}. We can state this result in terms of a free action of the cyclic group of order $2$ in the following way: {\it let $\z_2=\{\overline{0}, \overline{1} \}$ be the cyclic group of order $2$, and let $\tau \colon\thinspace \z_2 \times \mathbb{S}^m \to \mathbb{S}^m$ be the free action defined by $\tau(\overline{i},x)=(-1)^i x$ for all $i\in \{ 0,1\}$ and $x\in M$. If $f \colon\thinspace \mathbb{S}^m \to \mathbb{R}^m$ is a continuous map, then there exists $x \in \mathbb{S}^m$ such that $f(\tau(\overline{0},x))=f(\tau(\overline{1},x))$.} A very natural generalisation of this theorem is to replace $\mathbb{S}^m$ and $\mathbb{R}^m$ by manifolds $M$ and $N$, and replace the antipodal map by a free involution on $M$. One may then ask if any continuous map from $M$ to $N$ collapses an orbit of this $\z_2$-action. We will refer to this question as the Borsuk-Ulam problem for $M$ and $N$. We focus our discussion on the case where $M$ and $N$ are surfaces, for which much progress has been made over the past fifteen years with respect to this problem. In 2005, classical techniques of algebraic topology were used to obtain a complete answer to the problem for every compact surface $M$ without boundary equipped with a free involution and $N=\mathbb{R}^2$~\cite{Gon}. A few years later, the problem in the case where $M$ and $N$ are compact surfaces without boundary was studied, and using the theory of surface braid groups, a complete answer to the problem was given~\cite{GonGua}. To our knowledge, this was the first time in the literature that these groups were used to study Borsuk-Ulam-type problems. In a similar vein, the Borsuk-Ulam problem was solved in the case where the domain is the product of two compact surfaces with boundary equipped with the diagonal involution, and the target is $\mathbb{R}^n$~\cite{GonSan} or a compact surface without boundary~\cite{GonSanSil}, braid theory being used in the latter case. 
More recently, another natural and more refined problem that generalises the Borsuk-Ulam theorem arose, namely
the question of the classification of the free homotopy classes of maps between $M$ and $N$, where $M$ is equipped with a free $\z_2$-action for which every representative map of a given homotopy class collapses an orbit of the action.
This problem was solved by the authors in three recent papers in the cases where $M$ and $N$ are compact surfaces without boundary of Euler characteristic zero~\cite{GonGuaLaa1,GonGuaLaa2,GonGuaLaa3}.
The results depend on the choice of free involution on $M$, and once more, braid theory plays an important r\^ole in the solution of the problem.
	
In this paper, our aim is to initiate the study of another generalisation of the Borsuk-Ulam theorem by considering free actions of groups other than $\z_2$ on spaces as follows. 
Let $M$ and $N$ be topological spaces, and let $G$ be a non-trivial group for which there exists a free action $\tau \colon\thinspace G \times M \to M$ of $G$ on $M$. Following~\cite{GonGuaLaa1,GonGuaLaa2,GonGuaLaa3}, we say that a free homotopy class $\beta \in [M,N]$ \emph{has the Borsuk-Ulam property with respect to the pair $(G,\tau)$} if for every representative map $f\colon\thinspace M \to N$ of $\beta$, there exist $g_1,g_2 \in G$, where $g_1\neq g_2$, and $x \in M$ such that $f(\tau(g_1,x))=f(\tau(g_2,x))$, and that the quadruple $(M,G,\tau;N)$ \emph{satisfies the Borsuk-Ulam property} if every homotopy class of maps between $M$ and $N$ has the Borsuk-Ulam property with respect to $(G,\tau)$. For technical reasons, we also define the Borsuk-Ulam property for pointed homotopy classes. More precisely, a pointed homotopy class $\alpha \in [M,m_1;N,n_1]$ \emph{has the Borsuk-Ulam property with respect to the pair $(G,\tau)$} if for every representative map $f\colon\thinspace (M,m_1) \to (N,n_1)$
of $\alpha$ there exist $g_1,g_2 \in G$, where $g_1\neq g_2$, and $x \in M$ such that $f(\tau(g_1,x))=f(\tau(g_2,x))$.
	
The first case that we study is that of free actions of a finite cyclic group on a compact surface $M$ without boundary, and where $N$ is the Euclidean plane $\rtwo$. Note that in this situation, there is just one homotopy class, and in this sense, our results generalise those of~\cite{Gon}. We underline the important r\^{o}le of braid theory in our work.

To state Theorem~\ref{th:main}, which is the main result of this paper, we first recall some facts and notation. Given a pathwise-connected CW-complex $M$ and a proper free cellular action $\tau \colon\thinspace G \times M \to M$, where $G$ is a non-trivial discrete group, we say that two points $x_1,x_2 \in M$ are equivalent if there exists $g \in G$ such that $\tau(g,x_1)=x_2$. Let $M_\tau$ denote the quotient space (or orbit space) of $M$ by this action, and let $p_\tau \colon\thinspace M \to M_\tau$ denote the natural projection of $M$ onto $M_{\tau}$. Then $p_\tau$ is a regular covering map, and the associated group of deck transformations is $G$. Thus the space $M_\tau$ is also a pathwise-connected CW-complex whose dimension is that of $M$, and we have the following short exact sequence (we omit the basepoints):
	\begin{equation}\label{seq:main}
		\xymatrix{
			\{1 \} \ar[r] & \pi_1 (M) \ar[r]^-{(p_\tau)_\#} & \pi_1(M_\tau) \ar[r]^-{\theta_\tau} & \dfrac{\pi_1(M_\tau)}{(p_\tau)_\#(\pi_1(M))} \cong G \ar[r] & \{1 \}, \\
	}\end{equation}
	where $\theta_{\tau}$ is the natural projection.
	
Suppose now that $G$ is an Abelian group. Then the homomorphism $\theta_\tau$ of~(\ref{seq:main}) factors through the Abelianisation $(\pi_1(M_\tau))_{\ab}$ of $\pi_1(M_\tau)$ that is isomorphic to the first homology group $H_{1}(M_{\tau},\z)$ of $M_\tau$ with integral coefficients. We denote this factorisation by $(\theta_\tau)_\ab \colon\thinspace H_1(M_\tau) \to G$. If $M_\tau$ is a compact, non-orientable surface without boundary, let $\delta$ denote the unique element of $H_1 (M_\tau)$ of order two. The main goal of this paper is to prove the following theorem.
	
\begin{theorem}\label{th:main}
Let $M$ be a compact surface without boundary, 
 and let $\tau \colon\thinspace \z_n \times M \to M$ be a free action. Then the quadruple $(M,\z_n,\tau;\rtwo)$ has the Borsuk-Ulam property if and only if the following conditions are satisfied:
\begin{enumerate}[(1)]
\item $n \equiv 2 \bmod 4$.
\item $M_\tau$ is non-orientable, and $( \theta_\tau)_\ab (\delta)$ is non trivial.
\end{enumerate}
\end{theorem}

Notice that for $M= \mathbb{S}^2$ and $n=2$, Theorem~\ref{th:main} is a special case of the classical Borsuk-Ulam theorem.

 Following Theorem~\ref{th:main}, it is natural to ask what happens for actions of finite non-cyclic groups. Our techniques may be can be applied in part to this more general case, and in Section~\ref{sec:exam_sigma1}, we illustrate their use by giving a solution to the Borsuk-Ulam problem in certain cases where $G$ is the symmetric group $\Sigma_n$, where $n\geq 2$. More precisely, in Proposition~\ref{prop:ex121}, we show that for all $n\geq 2$, there exist actions of $\Sigma_n$ on two compact surfaces $M_1$ and $M_2$ without boundary. In the case of $M_{1}$, the Borsuk-Ulam property does not hold for maps into $\mathbb{R}^2$, while in the case of $M_{2}$, we do not know the answer in general.
	
Apart from this Introduction, this paper has four sections. In Theorem~\ref{th:borsuk_braid} of Section~\ref{sec:preli}, we give an algebraic criterion involving configuration spaces and braid groups that enables us to decide whether a given homotopy class of maps from $M$ to $N$ has the Borsuk-Ulam property, where $M$ is a connected manifold equipped with a free action of an Abelian group $G$, and $N$ is either the Euclidean plane $\mathbb{R}^{2}$ or a compact surface without boundary different from the $2$-sphere $\mathbb{S}^{2}$ and the real projective plane $\mathbb{R}P^{2}$. 
 One
implication of the condition of  Theorem~\ref{th:borsuk_braid} is
valid in general, while for the other implication, at present, we only
know that it holds in the case where $G$ is Abelian. This is the
reason for which we do not know presently if the surface $M_{2}$ of
the previous paragraph satisfies the Borsuk-Ulam property. In Section~\ref{sec:braid_group}, in Theorem~\ref{th:bzn_rtwo} we exhibit a presentation for the braid group that arises in the statement of Theorem~\ref{th:borsuk_braid} in the case where $G$ is cyclic of order $n$ and $N=\mathbb{R}^{2}$. Section~\ref{sec:proof} is devoted to the proof of Theorem~\ref{th:main}.  Finally in Section~\ref{sec:exam_sigma1}, we discuss the above-mentioned examples of two surfaces on which $\Sigma_n$ acts freely with regard to the Borsuk-Ulam property.

	
	\section{Preliminaries and Generalities}\label{sec:preli}
	
	We start by stating and proving the following result that is an algebraic and homotopy theoretical characterisation of the set of pointed homotopy classes of maps between pairs of $CW$-complexes. It generalises~\cite[Theorem~4]{GonGuaLaa1} in some sense, and it will be used in the proof of Lemma~\ref{lem:equiv_maps}.
	
	\begin{theorem}\label{th:homotopy_algebra}
		Let $(M,m_1)$ and $(N,n_1)$ be pointed pathwise-connected $CW$-complexes such that $\pi_i(N,n_1)$ is trivial for $2 \leq i \leq \dim(M)$. Then the map $\Gamma\colon\thinspace  [M,m_1;N,n_1] \to \hom( \pi_1(M,m_1),\pi_1(N,n_1))$ defined by $\Gamma \left( [f] \right)=f_\#$ is a bijection.
	\end{theorem}
	
	\begin{proof}
		Let $\widetilde{N}=K(\pi_1(N),1)$ be the first stage of the Postnikov tower of $N$, and let $h\colon\thinspace N \to \widetilde{N}$ be a map that induces an isomorphism on the level of fundamental groups. By the definition of $h$ and the hypothesis, the homomorphism $h_\#\colon\thinspace \pi_i(N,n_1) \to \pi_i(\widetilde{N},h(n_1))$ is an isomorphism if $1 \leq i \leq \dim(M)$, and it is surjective if $i=\dim(M) +1$. From these observations, standard arguments of homotopy theory and~\cite[Chapter~VI, Theorems~6.12 and~6.13]{White}, the map $\Gamma_1 \colon\thinspace [M,m_1;N,n_1] \to [M,m_1;\widetilde{N},h(n_1)] $ defined by $\Gamma_1\left( [f] \right)=[h \circ f]$ is a bijection. Further, the map $\Gamma_2 \colon\thinspace [M,m_1;\widetilde{N},h(n_1)] \to \hom( \pi_1(M,m_1),\pi_1(\widetilde{N},h(n_1)))$ given by $\Gamma_2 \left( [g] \right)=g_\#$ is a bijection by~\cite[Chapter~V, Theorem~4.3]{White}. Now let $\Gamma_3 \colon\thinspace \hom( \pi_1(M,m_1),\pi_1(\widetilde{N},h(n_1))) \to \hom( \pi_1(M,m_1),\pi_1(N,n_1))$ be the map defined by $\Gamma_3 \left( \varphi \right)=\left(h_\#\right)^{-1} \circ \varphi$. Then $\Gamma_3$ is a bijection, and the result follows by taking $\Gamma=\Gamma_3 \circ \Gamma_2 \circ \Gamma_1$.
	\end{proof}
	
	Let $M$ and $N$ be pathwise-connected $CW$-complexes, and let $G$ be a non-trivial discrete group for which there exist proper free cellular actions $\tau \colon\thinspace G \times M \to M$ and $\tau_1 \colon\thinspace G \times N \to N$. A map $f \colon\thinspace M \to N$ is said to be \emph{$(\tau,\tau_1)$-equivariant} if  $f(\tau(g,x))=\tau_1(g,f(x))$ for all $g \in G$ and all $x \in M$. Given a pointed homotopy class $\alpha =[f]\in [M,m_1;N,n_1]$, let $\alpha_\#=f_\# \colon\thinspace \pi_1(M,m_1) \to \pi_1(N,n_1)$. Note that the homomorphism $\alpha_\#$ is independent of the choice of the representative map $f$ of $\alpha$. The following result provides an algebraic criterion to decide whether a pointed homotopy class has an equivariant representative map, and it generalises~\cite[Lemma 5]{GonGuaLaa1}.
	
	\begin{lemma}\label{lem:equiv_maps}
		Let $(M,m_1)$ and $(N,n_1)$ be pointed pathwise-connected $CW$-complexes (resp.\ compact surfaces without boundary) such that $\pi_i (N, n_1)$ is trivial for $2\leq i\leq \dim(M)$, let $G$ be a non-trivial group for which there exist proper free cellular actions $\tau \colon\thinspace G \times M \to M$ and $\tau_1 \colon\thinspace G \times N \to N$, and let $\alpha \in [M, m_1; N, n_1]$ be a pointed homotopy class. Then the following conditions are equivalent:
		\begin{enumerate}
			\item\label{it:equiv_maps1} there exists a representative map (resp.\ homeomorphism) $f \colon\thinspace (M, m_1) \to (N, n_1)$ of $\alpha$ that is $(\tau, \tau_1)$-equivariant.
			
			\item\label{it:equiv_maps2} there exists a homomorphism (resp.\ isomorphism) $\psi \colon\thinspace  \pi_1 (M_{\tau}, p_{\tau}(m_1)) \to \pi_1 (N_{\tau_1}, p_{\tau_1} (n_1))$ such that the following diagram is commutative: 
			\begin{equation}\label{diag:diag_equiv}\begin{gathered}\xymatrix{
						\pi_1(M,m_1) \ar[rr]^-{\alpha_\#} \ar[d]_-{(p_\tau)_\#}  & & \pi_1 (N, n_1) \ar[d]^-{( p_{\tau_1})_\#} \\
						\pi_1(M_\tau, p_\tau(m_1)) \ar@{.>}[rr]^-{\psi} \ar[rd]_-{\theta_\tau} &  & \pi_1 (N_{\tau_1}, p_{\tau_1} (n_1)) \ar[ld]^-{\theta_{\tau_1}}  \\
						& G .&
			}\end{gathered}\end{equation}
		\end{enumerate}
	\end{lemma}
	

	\begin{proof}
		We first prove that condition~(\ref{it:equiv_maps1}) implies condition~(\ref{it:equiv_maps2}). Let $f\colon\thinspace  (M, m_1) \to (N,n_1)$ be a $(\tau, \tau_1)$-equivariant representative map of $\alpha$. Then the map $\overline{f}\colon\thinspace (M_\tau,p_\tau(m_1)) \to (N_{\tau_1}, p_{\tau_1} (n_1))$ between the corresponding orbit spaces defined by $\overline{f}(y)= p_{\tau_1}\circ f(x)$ for all $y\in M_{\tau}$, where $x\in p_{\tau}^{-1}(y)$, is well defined, and the following diagram is commutative:
		\begin{equation}\label{diag:diag_equiv_aux}\begin{gathered}\xymatrix{
					(M, m_1) \ar[r]^-{f} \ar[d]_-{p_\tau} & (N,n_1) \ar[d]^-{p_{\tau_1}} \\
					(M_\tau, p_\tau(m_1)) \ar[r]^-{\overline{f}} & (N_{\tau_1},
					p_{\tau_1} (n_1)).
		}\end{gathered}\end{equation}
		Let $\psi \colon\thinspace \pi_1 (M_\tau, p_\tau(m_1)) \to \pi_1 (N_{\tau_1}, p_{\tau_1} (n_1))$ denote the homomorphism induced by $\overline{f}$. The diagram~(\ref{diag:diag_equiv_aux}) then gives rise to the upper square of the diagram~(\ref{diag:diag_equiv}). It remains to show that $\theta_{\tau}=\theta_{\tau_{1}}\circ \psi$. To see this, let $a \in \pi_1(M_\tau, p_\tau(m_1))$, let $\alpha$ be a loop in $M_{\tau}$ based at $p_{\tau}(m_{1})$ such that $a=[\alpha]$, let $\alpha'$ be the loop in $M_{\tau}$ based at $p_{\tau}(m_{1})$ given by $\alpha'=\overline{f}\circ \alpha$, let $g=\theta_{\tau}(a)$, and let $\xi\colon\thinspace [0,1]\to M$ be the lift of $\alpha$ to $M$ for which $\xi(0)=m_{1}$. By commutativity of~(\ref{diag:diag_equiv_aux}), $\alpha'=\overline{f}\circ p_{\tau}\circ \xi=p_{\tau_{1}}(f\circ \xi)$, and since $f(\xi(0))=f(m_{1})=n_{1}$, $f\circ \xi\colon\thinspace [0,1]\to N$ is the lift to $N$ of $\alpha'$ whose initial point is $n_{1}$. Using the construction of the short exact sequence~(\ref{seq:main}), we have $\xi(1)=\tau(g,m_{1})$. It follows from the fact that $f$ is $(\tau,\tau_{1})$-equivariant that $f\circ \xi(1)=f(\tau(g,\xi(1)))=\tau_{1}(g,f\circ\xi(1))$, and so $\theta_{\tau_{1}}([\alpha'])=g$, from which we see that $\theta_{\tau_{1}}\circ \psi([\alpha])=\theta_{\tau_{1}}([\alpha'])= g=\theta_{\tau}(g)$ as required. Further, if $f$ is a homeomorphism, then $\overline{f}$ is too, and $\psi=\overline{f}_\#$ is an isomorphism. 

		
		
		We now prove that condition~(\ref{it:equiv_maps2}) implies condition~(\ref{it:equiv_maps1}). Since $p_{\tau_1}\colon\thinspace N \to N_{\tau_1}$ is a covering map, the triviality of $\pi_i (N, n_1)$ for $2 \leq i \leq \dim(M) $ implies that of  $\pi_i (N_{\tau_1}, p_{\tau_1}(n_1)) $ for all $2 \leq i \leq \dim(M)$. It follows from Theorem~\ref{th:homotopy_algebra} that there exists a map $\overline{f}\colon(M_\tau, p_\tau(m_1)) \to (N_{\tau_1}, p_{\tau_1} (n_1))$ such that $\overline{f}_\#=\psi$, and so by~(\ref{diag:diag_equiv}), there exists a lift $f\colon\thinspace (M, m_1) \to (N, n_1)$ of the map $\overline{f} \circ p_\tau$ for the covering $p_{\tau_1}$, from which we obtain the commutative diagram~(\ref{diag:diag_equiv_aux}). The injectivity of $(p_{\tau_1})_{\#}$ implies that $\alpha_\#=f_\#$, and so $f$ is a representative map of $\alpha$ by Theorem~\ref{th:homotopy_algebra}. Let us show that $f$ is $(\tau,\tau_1)$-equivariant. We will do this in two steps. We first claim that $f(\tau(g,m_1))= \tau_1(g,f(m_1))$ for all $g \in G$. To see this, let $\gamma \colon\thinspace [0,1] \to M$ be a path such that $\gamma(0)=m_1$ and $\gamma(1)=\tau(g,m_1)$. Then $\alpha=p_{\tau} \circ \gamma$ is a loop in $M_{\tau}$  based at $p_{\tau}(m_1)$ such that $\theta_{\tau}( [\alpha])=g$. The commutative diagrams~(\ref{diag:diag_equiv}) and~(\ref{diag:diag_equiv_aux}) imply that $\overline{f} \circ \alpha$ is a loop in $N_{\tau_1}$ based at $p_{\tau_1}(n_1)$, and that:
		\begin{equation*}
			\theta_{\tau_1}( [\overline{f} \circ \alpha])
			= (\theta_{\tau_1} \circ \overline{f}_\#)([\alpha])
			= \theta_{\tau}([\alpha])=g.
		\end{equation*}
		So if $\gamma_1 \colon\thinspace [0,1] \to N$ is the lift of $\overline{f} \circ \alpha$ by the covering $p_{\tau_1}$ such that $\gamma_1(0)=n_1$, then $\gamma_1(1)=\tau_1(g,n_1)$. Now $(f \circ \gamma)(0)=f(m_1)=n_1$ and:
		\begin{equation*}
			p_{\tau_1} \circ (f \circ \gamma)
			= \overline{f} \circ p_{\tau} \circ \gamma
			= \overline{f} \circ \alpha
			= p_{\tau_1} \circ \gamma_1.
		\end{equation*}
		It follows from the path lifting property that $f\circ \gamma =\gamma_1$, and hence:
		\begin{equation}\label{eq:equiv_m1}
			f(\tau(g,m_1))
			= (f \circ \gamma)(1)
			= \gamma_1(1)
			= \tau_1(g,n_1)
			= \tau_1(g,f(m_1)),
		\end{equation}
		which proves the claim. Now let $x\in M$, let $g\in G$, let $\xi \colon\thinspace [0,1] \to M$ be a path from $m_1$ to $x$, and let $\xi_1,\xi_2 \colon\thinspace [0,1] \to N$ be the paths defined by $\xi_1(t)=f(\tau(g,\xi(t)))$ and $\xi_2(t)=\tau_1 (g,f \circ \xi(t))$. Then:
		\begin{equation}\label{eq:xiinit}
			\xi_1(0)
			= f(\tau(g,\xi(0)))
			= f(\tau(g,m_1))
			\stackrel{(\ref{eq:equiv_m1})}{=} \tau_1(g,f(m_1))
			= \tau_1 (g,f \circ \xi(0))
			= \xi_2(0).
		\end{equation}
		Using~(\ref{diag:diag_equiv_aux}) and the definition of $p_{\tau}$ and $p_{\tau_{1}}$, for all $t\in [0,1]$, we have:
		\begin{align*}
			p_{\tau_{1}}\circ \xi_{1}(t)&=p_{\tau_{1}}\circ f(\tau(g,\xi(t)))=\overline{f}\circ p_{\tau}(\tau(g,\xi(t)))=\overline{f}\circ p_{\tau}(\xi(t))= p_{\tau_{1}}(f \circ \xi(t))= p_{\tau_{1}}(\tau_1 (g,f \circ \xi(t)))\\
			&= p_{\tau_{1}}\circ \xi_{2}(t).
		\end{align*}
		Thus $\xi_1$ and $\xi_2$ are lifts of the path $p_{\tau_1} \circ \xi_1$ by the covering $p_{\tau_1}$ that have the same initial point by~(\ref{eq:xiinit}), and therefore $\xi_1 =\xi_2$ by the path lifting property. From this, we conclude that:
		\begin{equation*}
			f(\tau(g,x))
			= f(\tau(g,\xi(1)))
			= \xi_1(1)
			= \xi_2(1)
			= \tau_1 (g,f \circ \xi(1))
			= \tau_1 (g,f (x)),
		\end{equation*}
		and $f$ is indeed $(\tau,\tau_1)$-equivariant. Finally, if $M$ and $N$ are compact surfaces without boundary and $\psi$ is an isomorphism, it follows from the classification theorem for surfaces and~\cite[Theorem~5.6.2]{Zies} that $\psi$ is induced by a homeomorphism. So we may take $\overline{f}$ to be a homeomorphism, in which case $f$ is also a homeomorphism.
	\end{proof}

	
	Let $S$ be either the Euclidean plane or a compact surface without boundary. Recall that $F_n(S)=\{(x_1,\ldots,x_n) \in S^{n} \, \mid \, \text{$x_i \neq x_j$ if $i \neq j$} \}$ is 
	the $n\textsuperscript{th}$ ordered configuration space of $S$. It is a pathwise-connected $2n$-dimensional manifold, and if $S$ is different from the $2$-sphere $\mathbb{S}^{2}$ or the projective plane $\mathbb{R}P^{2}$, then by~\cite[Corollary~2.2]{FadNeu}, 
	$\pi_i(F_n(S))$ is trivial for all $i \geq 2$.
	The fundamental group $\pi_1(F_n(S))$, denoted by $P_n(S)$, is the $n$-string pure braid group of $S$. Let $\Sigma_n$ be the symmetric group on the set $\{1, \ldots, n \}$. We adopt the convention that a product of permutations is read from left to right. Note that the map $\tau_1 \colon\thinspace \Sigma_n \times F_n(S) \to F_n(S)$ defined by 
	$\tau_1(\sigma,(x_1,\ldots,x_n))=(x_{\sigma^{-1}(1)},\ldots,x_{\sigma^{-1}(n)})$ is a free action. 
	Using the notation and the constructions given in the Introduction, $F_n(S)_{\tau_1}$ is the orbit space corresponding to this action, and the group $\pi_1(F_n(S)_{\tau_1})$, which we will denote by $B_n(S)$, is the \emph{$n$-string (full) braid group of the surface $S$}. Let $\iota_1 =(p_{\tau_1})_\#\colon\thinspace P_n(S) = \pi_1 (F_n(S)) \to \pi_1(F_n(S)_{\tau_{1}}) = B_n(S)$  and let $\pi_1=\theta_{\tau_1}\colon\thinspace B_n(S) =  \pi_1(F_n(S)_{\tau_{1}}) \to \pi_1(F_n(S)_{\tau_{1}})/\im{(p_{\tau_{1}})_\#}\cong \Sigma_n$. Now let $G$ be a finite group of order $n$, and let $i \colon\thinspace G \to \Sigma_n$ be a monomorphism. Let $\tau_2\colon\thinspace G \times F_n(S) \to F_n(S)$ be the free involution defined by $\tau_2(g,x) = \tau_1(i(g),x)$. We may regard $G$ as a subgroup of $\Sigma_n$, and $\tau_2$ as the restriction of $\tau_1$ to $G \times F_n(S)$. Let the fundamental group of the corresponding orbit space $F_n(S)_{\tau_2}$ be denoted by $B_G(S)$, and let $\iota_2=(p_{\tau_2})_{\#}\colon\thinspace P_n(S) = \pi_1 (F_n(S)) \to \pi_1(F_n(S)_{\tau_{2}}) = B_G(S)$ and $\pi_2=\theta_{\tau_2}\colon\thinspace \pi_1(F_n(S)_{\tau_{2}}) \to \pi_1(F_n(S)_{\tau_{2}})/\im{(p_{\tau_{2}})_\#}\cong G$. For $j,k \in \{ 1,2 \}$, we may regard $i$ and $\iota_j$ as inclusions, $\pi_k$ as a natural projection, and $\tau_2$ as the restriction of $\tau_1$ to $G \times F_n(S)$. The map $h\colon\thinspace F_n(S)_{\tau_2} \to F_n(S)_{\tau_1}$ that sends a $\tau_2$-orbit to the corresponding $\tau_1$ orbit is well defined, satisfies $h\circ p_{\tau_1}=p_{\tau_2}$, and
induces an injective homomorphism $h_{\#}\colon\thinspace \pi_1(F_n(S)_{\tau_2}) \to \pi_1(F_n(S)_{\tau_1})$. Using the construction of~(\ref{seq:main}), we obtain the following commutative diagram of short exact sequences:
	\begin{equation}\label{eq:diag_braids}
		\begin{gathered}
			\xymatrix{
				1 \ar[r] & P_n (S) \ar[r]^-{\iota_2} \ar@{=}[d] & B_G(S) \ar[r]^-{\pi_2} \ar@{^{(}->}[d]_-{h_{\#}} & G \ar[r] \ar@{^{(}->}[d]^-{i} & 1\\
				1 \ar[r] & P_n (S) \ar[r]^-{\iota_1} & B_n (S) \ar[r]^-{\pi_1} & \Sigma_n \ar[r] & 1.
			}
		\end{gathered}
	\end{equation}
	Identifying $G$ with $i(G)$ and $\pi_1(F_n(S)_{\tau_2})=B_G(S)$ with $\im{h_{\#}}$, we see that $B_G(S)\subset \pi_1^{-1}(G)$, and it follows from~(\ref{eq:diag_braids}) that $B_G(S)= \pi_1^{-1}(G)$.

	\begin{remark}\label{remark:cayley}
		Given a finite group $G = \{ g_1, g_2, \ldots, g_n \}$ of order $n$, the set of bijections $\Sigma_G$ from $G$ to $G$ equipped with the operation of composition is a group that is isomorphic to $\Sigma_n$, and that we will identify in what follows. 
		 Note that composition in $\Sigma_G$ is read from right to left, but as we mentioned above, a product of permutations will be read from left to right. For all $g \in G$, the map $i_g\colon\thinspace G \to G$ given by $i_g(g') = g\ldotp g'$ for all $g'\in G$ is a bijection, and the map $i\colon\thinspace G \to \Sigma_n$ that assigns $g$ to $i_g$ is an embedding of $G$ in $\Sigma_n$. With respect to the homomorphism $i$, and with the notation of the previous paragraph, $F_n(S)$ may be described as $F_n(S)=\{(x_{g_1},\ldots,x_{g_n}) \in S^{n} \, \mid \, \text{$x_{g_j} \neq x_{g_k}$ if $g_j \neq g_k$} \}$, and $\tau_2(g,(x_{g_1},\ldots,x_{g_n})) = (x_{g\ldotp g_1},\ldots,x_{g\ldotp g_n})$.
	\end{remark}

	As for the Artin braid groups, the elements of $B_n(S)$ may be described geometrically as follows (other equivalent geometric constructions may be found in~\cite{Bir,Han,Mura} for example). Pick $n$ distinct points $P_1, \ldots, P_n$ in $S$. A geometric $n$-braid is a $n$-tuple $\beta=(b_1,\ldots,b_n)$ where $b_1,\ldots,b_n\colon\thinspace [0,1] \to S \times [0,1]$ are paths (called strings of $\beta$) satisfying the following conditions:
	\begin{enumerate}[(i)]
		\item $b_i(0)=P_i \times \{1\}$ for all $i=1,\ldots,n$.
		\item there exists a permutation $\widetilde{\pi}(\beta) \in \Sigma_n$, called the permutation of $\beta$, such that $b_i(1) =P_{\widetilde{\pi}(\beta)(i)} \times\{0\}$ for all $i=1,\ldots,n$.
		\item for all $t\in [0,1]$, the intersection $(S \times \{t \}) \cap \{b_1(t), \ldots, b_n(t) \}$ consists of $n$ distinct points.
	\end{enumerate}
	Two geometric $n$-braids $\beta_0$ and $\beta_1$ are said to be equivalent if there exists a continuous family of paths $b_1^s, \ldots, b_n^s: [0,1] \to S \times [0,1]$, $s \in [0,1]$, such that $h_0=\beta_0$, $h_1=\beta_1$, and $h_s=(b_1^s, \ldots, b_n^s)$ is an $n$-braid for all $s\in [0,1]$. This defines an equivalence relation on the set of geometric $n$-braids, and the equivalence class of a geometric $n$-braid is called simply an $n$-braid. Notationally, we do not distinguish a geometric $n$-braid from its equivalence class. If $\beta=(b_1, \ldots, b_n)$ and $\gamma=(c_1,\ldots,c_n)$ are geometric $n$-braids, then we define the product of $\beta$ and $\gamma$ by $\beta  \gamma=(b_1 \ast c_{\widetilde{\pi}(\beta)(1)},\ldots,b_n \ast c_{\widetilde{\pi}(\beta)(n)})$, where the symbol $\ast$ denotes concatenation of paths. This induces a group structure on the set of $n$-braids, and the resulting group is isomorphic to $B_n(S)$. Up to this isomorphism, we may identify $\pi_1$ and $\widetilde{\pi}$. The pure braid group $P_n(S)$ (resp.\ the group $B_G(S)$) consists of those $n$-braids whose permutation is the identity (resp.\ belongs to the subgroup $G$ of $\Sigma_n$).
	
	
	The following theorem is the main result of this section. It establishes a connection between the Borsuk-Ulam property and the braids groups of a surface, thus generalising~\cite[Theorem~7]{GonGuaLaa1}.

	\begin{theorem}\label{th:borsuk_braid}
		Let $(M, m_1)$ be a pointed, pathwise-connected $CW$-complex, and suppose that there exists a proper free cellular action $\tau\colon\thinspace G \times  M \to M$, where $G$ is a finite group of order $n \geq 2$ that we embed in $\Sigma_n$ as in Remark~\ref{remark:cayley}. Let $(S, s_1)$ be a pointed surface, where $S$ is either $\mathbb{R}^2$  or a compact surface without boundary different from $\mathbb{S}^2$ and $\mathbb{RP}^2$.
		
		\begin{enumerate}
			\item\label{it:borsuk_braid1} Let $\alpha \in [M,m_1;S,s_1]$ be a pointed homotopy class. Suppose that there exist homomorphisms $\varphi\colon\thinspace  \pi_1 (M, m_1) \to P_n (S)$ and $\psi\colon\thinspace \pi_1 (M_\tau, p_\tau(m_1)) \to B_G(S) $ for which the following diagram is commutative:
			\begin{equation}\label{eq:diag_borsuk_braid}\begin{gathered}\xymatrix{
						\pi_1 (M,m_1) \ar@{.>}[rr]^{\varphi} \ar[d]_{(p_\tau)_\#} \ar@/^0.9cm/[rrrr]^{\alpha_\#} 
						&& P_n(S) \ar[d]^-{\iota_2} \ar[rr]^{(p_1)_\#} && \pi_1(S,s_1) \\
						\pi_1(M_\tau,p_\tau(m_1)) \ar@{.>}[rr]^{\psi} \ar[rd]_{\theta_\tau} && B_G(S) \ar[ld]^{\pi_2} && \\
						& G,& & &
			}\end{gathered}\end{equation}
			where $p_1\colon\thinspace F_n(S) \to S$ denotes projection onto the first coordinate. Then $\alpha$ does not have the Borsuk-Ulam property with respect to $(G,\tau)$. Conversely, if $G$ is an Abelian group and $\alpha$ does not have the Borsuk-Ulam property with respect to $(G,\tau)$, then there exist homomorphisms $\varphi\colon\thinspace  \pi_1 (M, m_1) \to P_n (S)$ and $\psi\colon\thinspace \pi_1 (M_\tau, p_\tau(m_1)) \to B_G(S) $ for which the diagram~(\ref{eq:diag_borsuk_braid}) is commutative.

			\item\label{it:borsuk_pointed_free} Let $\beta\in [M, S]$ be a free homotopy class, and let $\alpha  \in [M, m_1; S, s_1]$ be a pointed homotopy class such that every representative map of $\alpha$ is also a representative map of $\beta$ when we forget the basepoints. If $\beta$ has the Borsuk-Ulam property with respect to $(G,\tau)$, then so does $\alpha$. Conversely, if $G$ is an Abelian group and $\alpha$ has the Borsuk-Ulam property with respect to $(G,\tau)$, then so does $\beta$.
		\end{enumerate}
	\end{theorem}
	
	\begin{remarks}\label{remark:main_th}\mbox{}
		%
		By applying~\cite[Chapter~V, Theorem~4.3 and Corollary~4.4]{White}, if $S$ is either $\mathbb{R}^2$ or a compact surface without boundary different from $\mathbb{S}^2$ and $\mathbb{RP}^2$ and $G$ is an Abelian group, Theorem~\ref{th:borsuk_braid} provides us with an algebraic criterion to decide whether a homotopy class has the Borsuk-Ulam property or not. The cases where $S= \mathbb{S}^2$ or $\mathbb{RP}^2$ is the subject of work in progress, and the techniques required are somewhat different. In the general case, for the time being we have not been able to remove the restriction that $G$ is Abelian in the converse of parts~(\ref{it:borsuk_braid1}) and~(\ref{it:borsuk_pointed_free}) in the statement of Theorem~\ref{th:borsuk_braid}. 
	\end{remarks}
	
	\begin{proof}[Proof of Theorem~\ref{th:borsuk_braid}]
		In what follows, for all $g \in G$ and $x \in M$, we will denote $\tau(g,x)$ simply by $gx$. With respect to Remark~\ref{remark:cayley}, by permuting the elements of $G$ if necessary, we may suppose that $g_1$ is its trivial element.
		
		\begin{enumerate}[(a)]
			\item 
			First suppose that there exist homomorphisms $\varphi\colon\thinspace  \pi_1 (M, m_1) \to P_n (S)$ and $\psi\colon\thinspace \pi_1 (M_\tau, p_\tau(m_1)) \to B_G(S) $ for which the diagram~(\ref{eq:diag_borsuk_braid}) is commutative. Since $F_n(S)$ is a $K(P_n(S),1)$-manifold, \cite[Chapter~V, Theorem~4.3]{White} and Lemma~\ref{lem:equiv_maps} imply that the homomorphism $\varphi$ is induced by a $(\tau, \tau_2)$-equivariant map $F\colon\thinspace  M \to F_n(S)$, so $F_{\#}=\varphi$. For $i=1,\ldots, n$, let $f_{g_i}\colon\thinspace  M \to S$ be a map such that:
			\begin{equation}\label{eq:F_compon}
				F (x)=(f_{g_1}( x),f_{g_2}(x),\ldots,f_{g_n}(x))
			\end{equation}
			for all $x \in M$. Since $F$ is $(\tau, \tau_2)$-equivariant, for all $i=1,\ldots,n$ and $x \in M$, by~(\ref{eq:F_compon}) and Reamrk~\ref{remark:cayley}, we have:
			\begin{equation}\label{eq:equiv_F}
				(f_{g_1}(g_i x),f_{g_2}(g_i x),\ldots,f_{g_n}(g_i x)) = F(g_ix)
				=\tau_2(g_i,F(x)) = (f_{g_i.g_1}( x),f_{g_i.g_2}(x),\ldots,f_{g_i.g_n}(x))
			\end{equation}
			Now let $f=f_{g_1}$. Since $g_1$ is the trivial element of $G$, we see from the first coordinate of~(\ref{eq:equiv_F}) that $f_{g_i}(x) = f(g_ix)$ for all $i=1,\ldots,n$ and $x \in M$. We conclude from~(\ref{eq:F_compon}) that
			\begin{equation*}
				F(x)=(f(x),f(g_2x),\ldots,f(g_nx)).
			\end{equation*}
			Since $F(x) \in F_n(S)$, then by definition, the pointed homotopy class $[f]$ does not have the Borsuk-Ulam property. Again, by~(\ref{eq:diag_borsuk_braid}) and \cite[Chapter~V, Theorem~4.3]{White}, we have $\alpha=[f]$, and thus $\alpha$ does not have the Borsuk-Ulam property as required.
			
Conversely, suppose  that $G$ is an Abelian group, and that $\alpha \in [M,m_1;S,s_1]$ does not have the Borsuk-Ulam property with respect to $(G,\tau)$. Then there exists a map 
			$f\colon\thinspace  (M, m_1) \to (S, s_1)$ such that $\alpha=[f]$, and for which the elements $f(g_1 x), \ldots, f(g_n x)$ are pairwise distinct for all $x \in M$. For $i=1,\ldots,n$, let $f_{g_i}\colon\thinspace M \to S$ be the map defined by $f_{g_i}(x) = f(g_ix)$. Then the map $F\colon\thinspace  M \to F_n(S)$ given by $F(x)=(f_{g_1} (x), \ldots, f_{g_n}(x))$ for all $x\in M$ is well defined, and it satisfies $p_1 \circ F=f$. We claim that $F$ is $(\tau,\tau_2)$-equivariant. Indeed, using the fact that $G$ is Abelian, for all $g \in G$ and $x \in M$, we have:
			%
			%
			\begin{align*}
				F(\tau(g,x))&=F(gx)=(f_{g_1}(gx),\ldots,f_{g_n}(gx))= (f(g_1(gx)),\ldots,f(g_n(gx)))\\
				&= (f((g \ldotp g_1)x),\ldots, f((g\ldotp g_n)x))
				= (f_{g\ldotp g_1}(x) ,\ldots, f_{g\ldotp g_n}(x)) \\
				& = \tau_2(g, (f_{g_1}(x),\ldots, f_{g_n}(x)))= \tau_2(g,F(x)).
			\end{align*}
			If $\varphi \colon\thinspace \pi_1 (M, m_1) \to P_n(S)$ is the homomorphism induced by $F$, then $(p_1)_\# \circ \varphi=(p_1)_\# \circ F_\#=(p_1 \circ F)_\#=(f_{g_1})_\#=f_\#=\alpha_\#$, and applying Lemma~\ref{lem:equiv_maps} to the pointed homotopy class $[F]$, there exists a homomorphism $\psi\colon\thinspace \pi_1 (M_\tau, p_{\tau}(m_1)) \to B_G(S)$ for which $\psi \circ (p_{\tau})_\#=\iota_2 \circ \varphi$ and $\pi_2 \circ \psi=\theta_\tau$, which yields the commutative diagram~(\ref{eq:diag_borsuk_braid}).
			
			\item 
Clearly, if $\beta$ has the Borsuk-Ulam property, then $\alpha$ does too because every representative map of $\alpha$ is also a representative map of $\beta$. Conversely, assume that $G$ is an Abelian group, and suppose that $\beta$ does not have the Borsuk-Ulam property. So there exists a map $f\colon\thinspace M \to S$ such that $\beta=[f]$ and for which the cardinality of the set $\{f(g_1 x), \ldots, f(g_n x) \}$ is equal to $n$ for all $x \in M$. By~\cite[Lemma~6.4]{Vick}, there exists a homeomorphism $H\colon\thinspace S \to S$  that is homotopic to the identity such that $H(f(m_1))=s_1$. Thus the map $H \circ f$ is a representative map of $\beta$, 
 the cardinality of the set $\{H \circ f(g_1 x), \ldots, H \circ f(g_n x) \}$ is equal to $n$ for all $x \in M$, and the pointed homotopy class $\alpha^\prime=[H \circ f] \in [M, m_1;S,s_1]$ does not have the Borsuk-Ulam property with respect to $(G,\tau)$  (as a pointed homotopy class). By part~(\ref{it:borsuk_braid1}), there exist homomorphisms $\varphi'\colon\thinspace \pi_1 (M, m_1) \to P_n (S)$ and $\psi'\colon\thinspace  \pi_1 (M_\tau, p_\tau(m_1)) \to B_G(S)$  such the following diagram is commutative:
			\begin{equation}\label{eq:diag_borsuk_braid_prime}\begin{gathered}\xymatrix{
						\pi_1 (M,m_1) \ar@{.>}[rr]^{\varphi^\prime} \ar[d]_{(p_\tau)_\#} \ar@/^0.9cm/[rrrr]^{\alpha^\prime_\#} 
						&& P_n(S) \ar[d]^-{\iota_2} \ar[rr]^{(p_1)_\#} && \pi_1(S,s_1) \\
						\pi_1(M_\tau,p_\tau(m_1)) \ar@{.>}[rr]^{\psi^\prime} \ar[rd]_{\theta_\tau} && B_G(S) \ar[ld]^{\pi_2} && \\
						& G .& & &
			}\end{gathered}\end{equation}
			%
			%
			%
			%
			By~\cite[Chapter~V, Theorem~4.3 and Corollary~4.4]{White}, there exists $\gamma \in \pi_1(S,s_1)$ such that the homomorphisms $\alpha_\#,\alpha^\prime_\# \colon\thinspace \pi_1(M,m_1) \to \pi_1(S,s_1)$ satisfying  $\alpha_\#(v)=\gamma \alpha^\prime_\#(v) \gamma^{-1}$ for all $v \in \pi_1(M,m_1)$. Further, by~\cite[Theorem~1.4]{Bir}, there exists $b \in P_n(S)$ such that $(p_1)_\#(b)=\gamma$. We define $\varphi\colon\thinspace  \pi_1 (M, m_1) \to P_n(S)$ (resp.\ $\psi\colon\thinspace  \pi_1 (M_\tau, p_\tau(m_1)) \to B_G(S) $) by $ \varphi(v)=b\varphi'(v)b^{-1}$ for all $v \in  \pi_1 (M, m_1)$ (resp.\ $\psi(w)=\iota_2(b)\psi^\prime (w)\iota_2 (b)^{-1}$ for all $w\in\pi_1 (M_\tau, p_\tau(m_1))$). Using the commutativity of~(\ref{eq:diag_borsuk_braid_prime}), it follows that the diagram~(\ref{eq:diag_borsuk_braid}) is also commutative, and therefore $\alpha$ does not have the Borsuk-Ulam property as a pointed homotopy class.\qedhere 
		\end{enumerate}
	\end{proof}
	
	The following corollary will be used in the proof of Propositions~\ref{prop:1}--\ref{prop:3},
	which is the restriction of Theorem~\ref{th:borsuk_braid} to the case where $S=\mathbb{R}^2$. 
	
	\begin{corollary}\label{cor:borsuk_braid_r2}
		Let $M$ be a compact surface without boundary, let $G$ be a finite Abelian group of order $n \geq 2$, and let $\tau\colon\thinspace G \times M \to M$ be a free action. Then $(M,G,\tau;\rtwo)$ does not have the Borsuk-Ulam property if and only if there exists a homomorphism $\psi \colon\thinspace \pi_1(M_\tau) \to B_{G}(\rtwo)$ such that the following diagram is commutative:
		\begin{equation}\label{diag:borsuk_r2}\begin{gathered}\xymatrix{
					\pi_1(M_\tau) \ar@{.>}[rr]^{\psi} \ar[rd]_{\theta_\tau} & & B_{G}(\rtwo) \ar[ld]^{\pi_2}  \\
					& G .& 
		}\end{gathered}\end{equation}
	\end{corollary}  
	
	\begin{proof} 
		First of all, since $\mathbb{R}^2$ is contractible then there is only one homotopy class $\alpha \in [M,  \mathbb{R}^2 ]$, and $\alpha_\# \colon\thinspace \pi_1(M) \to \pi_1(\mathbb{R}^2)$ is trivial. So by Theorem~\ref{th:borsuk_braid}, $(M,G,\tau;\rtwo)$ does not have the Borsuk-Ulam property if and only if there exist homomorphisms $\varphi\colon\thinspace  \pi_1 (M) \to P_n (\mathbb{R}^2)$ and $\psi\colon\thinspace \pi_1 (M_\tau) \to B_{G}(\mathbb{R}^2) $ for which the diagram~(\ref{eq:diag_borsuk_braid}) is commutative. In particular, if $(M,G,\tau;\rtwo)$ does not have the Borsuk-Ulam property then the lower triangle of~(\ref{eq:diag_borsuk_braid}) is the commutative diagram~(\ref{diag:borsuk_r2}). Conversely, given the commutative diagram~(\ref{diag:borsuk_r2}), the exactness of~(\ref{seq:main}) and the upper row of~(\ref{eq:diag_braids}) allow us to obtain~(\ref{eq:diag_borsuk_braid}).
	\end{proof}

	\section{The group $B_{\zn}(\rtwo)$}\label{sec:braid_group}

	Let $n\geq 2$. We denote the cyclic group of order $n$ by $\z_n = \{ \overline{0}, \overline{1},\ldots, \overline{n-1} \}$. With respect to Remark~\ref{remark:cayley} and the paragraph that precedes it, we fix an embedding $i\colon\thinspace \z_n \to \Sigma_n$ of $\zn$ in $\Sigma_n$, given by $i(\overline{1}) = \left( 1,n,\ldots,2 \right)$, and we identify $B_{\zn}(S)$ with its image in $B_n(S)$ under $h_{\#}$.
	 The aim of this section is to give a presentation of the group $B_{\z_n}(\rtwo)$. Recall that if $n\in \mathbb{N}$, the Artin braid group $B_n$ on $n$ strings is generated by the \emph{Artin generators} $\sigma_1,\ldots, \sigma_{n-1}$, subject to the following \emph{Artin relations}:
	\begin{align}
		\sigma_{i} \sigma_{j}&= \text{$\sigma_{j} \sigma_{i}$ if $\lvert i-j \rvert \geq 2$, where $1\leq i,j\leq n-1$}\label{eq:rel1}\\
		\sigma_{i} \sigma_{i+1} \sigma_{i}&= \text{$\sigma_{i+1} \sigma_{i} \sigma_{i+1}$ for $i=1,\ldots,n-2$.}\label{eq:rel2}
	\end{align}
	%
	%
	Geometrically, $\sigma_i$ and its inverse correspond to the braids illustrated in Figure~\ref{fig:sigma_i}.
	%
	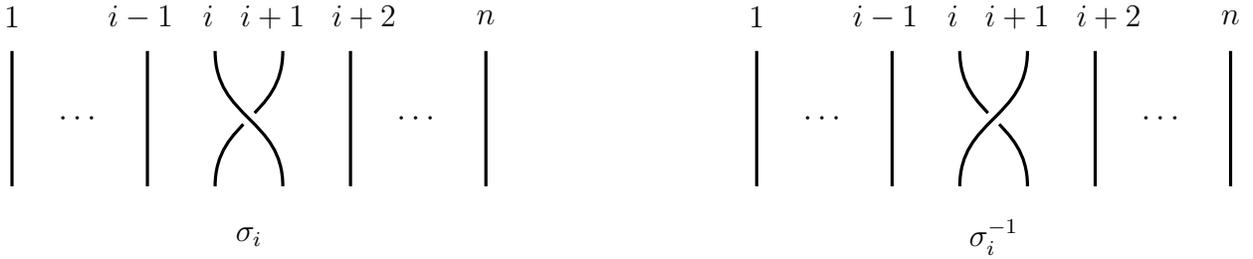
\begin{figure}[h]
		\hfill
		\begin{tikzpicture}[scale=0.9, very thick]
			
			\foreach \k in {5}
			{\draw (\k,3) .. controls (\k,2) and (\k-1,2) .. (\k-1,1);}
			
			\foreach \k in {4}
			{\draw[white,line width=6pt] (\k,3) .. controls (\k,2) and (\k+1,2) .. (\k+1,1);
				\draw (\k,3) .. controls (\k,2) and (\k+1,2) .. (\k+1,1);}
			
			\foreach \k in {15}
			{\draw (\k,3) .. controls (\k,2) and (\k+1,2) .. (\k+1,1);}
			
			\foreach \k in {16}
			{\draw[white,line width=6pt] (\k,3) .. controls (\k,2) and (\k-1,2) .. (\k-1,1);
				\draw (\k,3) .. controls (\k,2) and (\k-1,2) .. (\k-1,1);}

			\foreach \k in {1,3,6,8,12,14,17,19}
			{\draw (\k,1)--(\k,3);}
			
			\foreach \k in {2,7,13,18}
			{\node at (\k,2) {$\cdots$};}
			
			\foreach \k in {1,12}
			{\node at (\k,3.5) {$1$};
				\node at (\k+1.9,3.5) {$i-1$};
				\node at (\k+2.9,3.52) {$i$};
				\node at (\k+3.85,3.5) {$i+1$};
				\node at (\k+5.2,3.5) {$i+2$};
				\node at (\k+7,3.5) {$n$};}
			
			\node at (4.5,0.25) {$\sigma_{i}$};
			\node at (15.5,0.25) {$\sigma_{i}^{-1}$};
			
		\end{tikzpicture}
		\hspace*{\fill}
		\caption{The braid $\sigma_{i}$ and its inverse.}\label{fig:sigma_i}
	\end{figure}
	
	
	\begin{theorem}\label{th:bzn_rtwo}
		The group $B_{\zn}(\rtwo)$ admits the following presentation:
		\begin{enumerate}
			\item[\textbf{generating set:}] $\{g\} \cup \{A_{i,j}\}_{1\leq i<j\leq n}$.
			\item[\textbf{relations:}]\mbox{}
			\begin{enumerate}[(I)]
				\item\label{it:relA} $A_{r,s}^{-1} A_{i,j}A_{r,s}= \begin{cases}
					A_{i,j} & \text{if $r<s<i<j$ or $i<r<s<j$}\\
					A_{r,j} A_{i,j}A_{r,j}^{-1} & \text{if $r<i=s<j$}\\
					A_{r,j} A_{s,j} A_{i,j}A_{s,j}^{-1} A_{r,j}^{-1} & \text{if $i=r<s<j$}\\
					A_{r,j} A_{s,j}A_{r,j}^{-1} A_{s,j}^{-1} A_{i,j}A_{s,j} A_{r,j} A_{s,j}^{-1} A_{r,j}^{-1} & \text{if $r<i<s<j$.}
				\end{cases}$
				
				\item\label{rel:II} $gA_{i,j}g^{-1}= \begin{cases}
					A_{i+1,j+1} & \text{if $1\leq i< j\leq n-1$}\\
					A_{1,n} \cdots A_{n-1,n} A_{1,i+1} A_{n-1,n}^{-1}\cdots A_{1,n}^{-1} & \text{if $1\leq i\leq n-1$ and $j=n$.}
				\end{cases}$
				
				\item\label{rel:IV} $g^n=\displaystyle \prod_{j=2}^{n}\left( \prod_{i=1}^{j-1} A_{i,j} \right)$.
			\end{enumerate}
		\end{enumerate}
	\end{theorem}
	
	\begin{remarks}\label{rems:gensbzn}\mbox{}
		\begin{enumerate}
			\item\label{it:gensbzna} In terms of the Artin generators of $B_n$, we take $g=\sigma_1 \cdots \sigma_{n-1}$, and:
			\begin{equation}\label{eq:defaij}
				\text{$A_{i,j}=\sigma_{j-1}\cdots \sigma_{i+1} \sigma_{i}^2 \sigma_{i+1}^{-1} \cdots \sigma_{j-1}^{-1}$ for $1\leq i<j\leq n$.}
			\end{equation}
			One may check using~(\ref{eq:rel1}) and~(\ref{eq:rel2}) that:
			\begin{equation}\label{eq:altdefaij}
				\text{$A_{i,j}=\sigma_{i}^{-1} \cdots \sigma_{j-2}^{-1} \sigma_{j-1}^2 \sigma_{j-2} \cdots \sigma_{i}$ for $1\leq i<j\leq n$.}
			\end{equation}
			Geometrically, $A_{i,j}$ corresponds to the braid illustrated in Figure~\ref{fig:Aij}. 
			
			\item\label{it:gensbznb} The element $g^n$ that appears in relation~(\ref{rel:IV}) is the well-known full-twist braid, denoted by $\Delta_n^2$, that generates the centre of $B_n(\rtwo)$ if $n\geq 3$. 
			
			\item\label{it:gensbznc} With respect to the lower row of~(\ref{eq:diag_braids}), we have $\pi_2(g)=(1,n,\ldots,2)$. Taking $S=\rtwo$ and $G=\zn$, where we identify $\zn$ with the cyclic subgroup $\langle (1,n,\ldots,2)\rangle$ of $\Sigma_n$ in the manner described at the beginning of this section, it follows that $g\in B_{\zn}(\rtwo)$. 
			
			\item\label{it:gensbznd} Considering the permutation homomorphism $\pi_2 \colon\thinspace B_{\zn}(\rtwo) \to \zn$ of the upper row of~(\ref{eq:diag_braids}), and up to the identification of $B_{\zn}(\rtwo)$ (resp.\ $\zn$) with $\im{h_{\#}}$ (resp.\ with $\im{i}$), we have $\pi_2(g)=\overline{1}$ and $\pi_2(A_{i,j})=\overline{0}$ for all $1 \leq i < j \leq n$. Further, for all $a \in B_{\zn}(\rtwo)$, there exist unique $0\leq m\leq n-1$ and $w\in P_n(\rtwo)$ for which $a=wg^m$.
		\end{enumerate}
	\end{remarks}

	%
	%
	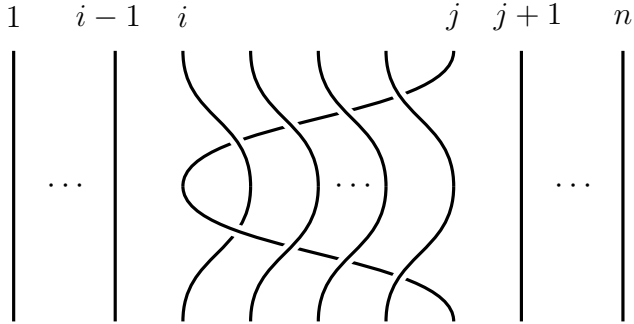
\begin{figure}[h]
		\hfill
		\begin{tikzpicture}[scale=0.9, very thick]
			\draw (7,6.5) .. controls (7,5.5) and (3,5.5) .. (3,4.5);
			\foreach \j in {3,...,6} 
			{\draw[white,line width=5pt] (\j,6.5).. controls (\j,5.5) and (\j+1,5.5) .. (\j+1,4.5);
				\draw (\j,6.5).. controls (\j,5.5) and (\j+1,5.5) .. (\j+1,4.5);
				\draw[white,line width=5pt] (3,4.5) .. controls (3,3.5) and (7,3.5) .. (7,2.5);}
			\draw (4,4.5) .. controls (4,3.5) and (3,3.5) .. (3,2.5);
			\draw[white,line width=5pt] (3,4.5) .. controls (3,3.5) and (7,3.5) .. (7,2.5);
			\draw (3,4.5) .. controls (3,3.5) and (7,3.5) .. (7,2.5);
			\foreach \j in {4,...,6} 
			{\draw[white,line width=5pt] (\j+1,4.5).. controls (\j+1,3.5) and (\j,3.5) .. (\j,2.5);
				\draw (\j+1,4.5).. controls (\j+1,3.5) and (\j,3.5) .. (\j,2.5);
			};
			\foreach \j in {0.5,2,8,9.5}
			{\draw (\j,6.5)-- (\j,2.5);}
			\foreach \k in {0.75,5,8.25}
			{\node at (\k+0.55,4.5) {$\cdots$};}
			\node at (0.5,7) {$1$};
			\node at (1.9,7) {$i-1$};
			\node at (3,7) {$i$};
			\node at (7,7) {$j$};
			\node at (8.1,7) {$j+1$};
			\node at (9.5,7) {$n$};
		\end{tikzpicture}
		\hspace*{\fill}
		\caption{The element $A_{i,j}$ of $B_{n}$.}\label{fig:Aij}
	\end{figure}
	%
	
	
	

	\begin{proof}[Proof of Theorem~\ref{th:bzn_rtwo}]
		We apply the method of~\cite[Chapter~10, Proposition~1]{Johnson} to the short exact sequence of the upper row of~(\ref{eq:diag_braids}), where we take $S= \rtwo$, $G=\zn$, and $g$ and $A_{i,j}$, where $1\leq i<j\leq n$, as in Remarks~\ref{rems:gensbzn}(\ref{it:gensbzna}). We identify the cyclic subgroup $\langle (1,n,\ldots,2)\rangle$ of $\Sigma_n$ with $\zn$ as in Remarks~\ref{rems:gensbzn}(\ref{it:gensbznc}). With respect to the upper row of~(\ref{eq:diag_braids}), it follows that $g$ is a coset representative of the generator $\overline{1}$ of $\zn$. Since $(A_{i,j})_{1\leq i<j\leq n}$ is a generating set of $P_n(\rtwo)$ by~\cite[Chapter~1, Lemma~4.2,]{Han}, it follows from~\cite[Chapter~10, Proposition~1]{Johnson} that $\{g\} \cup \{A_{i,j}\}_{1\leq i<j\leq n}$ is a generating set for $B_{\zn}(\rtwo)$. From the same proposition, the following three types of relation form a complete set of relations of $B_{\zn}(\rtwo)$:
		\begin{enumerate}[(i)]
			\item the relations emanating from $P_n(\rtwo)$. By~\cite[Chapter~1, Lemma~4.2]{Han}, these are the relations~(\ref{it:relA}) of the statement.
			
			\item the conjugates of the generators of $P_n(\rtwo)$ by $g$, rewritten in terms of the generators of $P_n(\rtwo)$. Suppose first that $1\leq i<j\leq n-1$. Using~(\ref{eq:rel1}) and~(\ref{eq:rel2}), we see that that $g\sigma_k g^{-1}=\sigma_{k+1}$ for all $k=1,\ldots, n-2$, and it follows from~(\ref{eq:defaij}) that $gA_{i,j}g^{-1}= A_{i+1,j+1}$, which is the first relation of~(\ref{rel:II}). Assume now that $1\leq i\leq n-1$ and $j=n$. One may check (by induction on $n$ for example) that $\sigma_{1}\cdots \sigma_{n-1} \sigma_{n-1}\cdots \sigma_1=A_{1,n} \cdots A_{n-1,n}$, and it follows from this and~(\ref{eq:altdefaij}) that:
			\begin{align*}
				gA_{i,n}g^{-1}&= \sigma_{1}\cdots \sigma_{n-1} \ldotp \sigma_{n-1}\cdots \sigma_{i+1} \sigma_{i}^2 \sigma_{i+1}^{-1} \cdots \sigma_{n-1}^{-1} \ldotp \sigma_{1}^{-1} \cdots \sigma_{n-1}^{-1}\\
				&= \sigma_{1}\cdots \sigma_{n-1} \sigma_{n-1}\cdots \sigma_1 \ldotp \sigma_1^{-1} \cdots \sigma_i^{-1} \sigma_{i}^2 \sigma_{i} \cdots \sigma_1 \ldotp \sigma_1^{-1} \cdots \sigma_{n-1}^{-1} \sigma_{1}^{-1} \cdots \sigma_{n-1}^{-1}\\
				&= A_{1,n} \cdots A_{n-1,n} \sigma_1^{-1} \cdots \sigma_{i-1}^{-1} \sigma_{i}^2 \sigma_{i-1} \cdots \sigma_1 A_{n-1,n}^{-1} \cdots A_{1,n}^{-1}\\
				&= A_{1,n} \cdots A_{n-1,n} A_{1,i+1} A_{n-1,n}^{-1} \cdots A_{1,n}^{-1},
			\end{align*}
			which is the second relation of~(\ref{rel:II}).
			
			\item the lift of the relation $n\ldotp \overline{1}=\overline{0}$ of $\zn$ in terms of generators of $P_n(\rtwo)$. Since $g^n=\Delta_n^2$ by Remarks~\ref{rems:gensbzn}(\ref{it:gensbznb}), relation~(\ref{rel:IV}) of the statement follows by noting that $\Delta_n^2=\prod_{j=2}^{n}\left( \prod_{i=1}^{j-1} A_{i,j} \right)$.\qedhere
		\end{enumerate}
	\end{proof}
	
	
	%
	
	
	\section{Proof of Theorem~\ref{th:main}}\label{sec:proof}
	
	Let $M$ be a compact surface without boundary, and let $\zn=\{\overline{0}, \overline{1}, \overline{2}, \ldots, \overline{n-1} \}$ denote the cyclic group of order $n \geq 2$. If $\tau \colon\thinspace \zn \times M \to M$ is a free action, to simplify the notation, we will consider $\tau$ to be the continuous map $\tau\colon\thinspace M \to M$ defined by $\tau =\tau(\overline{1},\cdot)$, that satisfies $\tau^n={\rm Id}$ and for which the cardinality of the set $\{x, \tau(x), \tau^2(x), \ldots, \tau^{n-1}(x) \}$ is equal to $n$ for all $x \in M$. Then the orbit space $M_\tau$ is also a compact surface without boundary. In what follows, we will distinguish the following three possibilities, and in each case, we define an element $\widehat{\delta}$ of $\pi_1 (M_\tau)$:
	\begin{enumerate}[(I)]
		\item\label{it:orient} $M_\tau$ is orientable, $\pi_1 (M_\tau)=\left\langle a_1,\ldots,a_{2m} \, \mid \, [a_1,a_2]\cdots [a_{2m-1},a_{2m}] \right\rangle$, and $\widehat{\delta}=1$. 
		
		\item\label{it:nonorientodd} $M_\tau$ is non-orientable of odd genus, $\pi_1 (M_\tau)=\left\langle c,a_1,a_2,\ldots,a_{2m-1},a_{2m} \, \mid \, c^2[a_1,a_2]\cdots [a_{2m-1},a_{2m}]\right\rangle$, and $\widehat{\delta}=c$.

		\item\label{it:nonorienteven} $M_\tau$ is non-orientable of even genus, $\pi_1 (M_\tau)= \left\langle u,v,a_1,\ldots,a_{2m} \, \mid \, uvuv^{-1}[a_1,a_2]\cdots [a_{2m-1},a_{2m}]\right\rangle$, and $\widehat{\delta}=u$.
	\end{enumerate}
	
	%
	
	\begin{remark}\label{rem:theta_tau}
		Let $G= \zn$, and consider the homomorphism $\theta_\tau\colon\thinspace \pi_1 (M_\tau) \to \zn$ defined by~(\ref{seq:main}). If $M$ is a compact surface without boundary, then $\theta_{\tau}(\widehat{\delta})=\overline{0}$ if $M_\tau$ is orientable or if $n$ is odd. In other words, if $\theta_\tau(\widehat{\delta})\neq \overline{0}$, then $M_\tau$ is non-orientable, $n$ is even, and $\theta_\tau (\widehat{\delta})=\overline{n/2}$. 
%
%
Since the target of the
homomorphism  $\theta_\tau\colon\thinspace \pi_1 (M_\tau) \to \zn$ is
Abelian, the homomorphism $\theta_\tau$ factors through the
Abelianisation $\pi_1(M_\tau)_{Ab}$ via the homomorphism that we denote $(\theta_\tau)_{Ab}\colon\thinspace \pi_1 (M_\tau)_{Ab} \to \z_n$, and for which $\theta_\tau(\widehat{\delta})=(\theta_\tau)_{Ab}(\delta)$.
	\end{remark}
	
	In order to establish Theorem~\ref{th:main}, by Remark~\ref{rem:theta_tau}, it suffices to prove the following three propositions.
	
	
	\begin{proposition}\label{prop:1} 
		Let $M$ be a compact surface without boundary that admits a free action $\tau$ of $\zn$. With the notation of~(\ref{it:orient})--(\ref{it:nonorienteven}) above, suppose that $\theta_\tau(\widehat{\delta})= \overline{0}$. Then $(M,\zn,\tau;\rtwo)$ does not have the Borsuk-Ulam property.
	\end{proposition}
	
	\begin{proposition}\label{prop:2}
		Let $M$ be a compact surface without boundary that admits a free action $\tau$ of $\zn$. With the notation of~(\ref{it:orient})--(\ref{it:nonorienteven}) above, suppose that $M_\tau$ is non-orientable, $n=4k$ for some $k \in \mathbb{N}$, and $\theta_{\tau}(\widehat{\delta})= \overline{2k}$. Then $(M,\zn,\tau;\rtwo)$ does not have the Borsuk-Ulam property.
	\end{proposition}
	
	\begin{proposition}\label{prop:3}
		Let $M$ be a compact surface without boundary that admits a free action $\tau$ of $\zn$. With the notation of~(\ref{it:orient})--(\ref{it:nonorienteven}) above, suppose that $M_\tau$ is non-orientable, $n=4k+2$ for some $k \in \mathbb{N}$, and $\theta_{\tau}(\widehat{\delta})= \overline{2k+1}$. Then $(M,\zn,\tau;\rtwo)$ has the Borsuk-Ulam property.
	\end{proposition}

	To prove Propositions~\ref{prop:1}--\ref{prop:3}, we will apply Corollary~\ref{cor:borsuk_braid_r2} to the case where $G$ is the finite cyclic group $\{\overline{0}, \overline{1}, \overline{2}, \ldots, \overline{n-1} \}$. In what follows, we shall identify $G$ with the subgroup of $\Sigma_n$ generated by the $n$-cycle 
	$(1,n,\ldots,2)$, and the generator $\overline{1}$ of $G$ with this $n$-cycle.
	
	\begin{proof}[Proof of Proposition~\ref{prop:1}]
		We will define a homomorphism $\psi \colon\thinspace \pi_1(M_\tau) \to B_{\zn}(\rtwo)$ for which the diagram~(\ref{diag:borsuk_r2}) is commutative, from which it will follow by Corollary~\ref{cor:borsuk_braid_r2} that $(M,\zn,\tau;\rtwo)$ does not have the Borsuk-Ulam property. With respect to the cases~(\ref{it:orient})--(\ref{it:nonorienteven}), let $\psi\colon\thinspace \pi_1(M_\tau) \to B_{\zn}(\rtwo)$ be the map defined on the generators of $\pi_1(M_\tau)$ as follows:
		\begin{enumerate}
			\item $\psi(a_i)= g^{b_i}$ for $i=1,\ldots, 2m$, where $g\in B_{\z_n}(\rtwo)$ 
			is as defined in Remark~\ref{rems:gensbzn}(\ref{it:gensbzna}), 
			 and $0\leq b_i\leq n-1$ is such that $\theta_\tau(a_i)= (1,n,\ldots,2)^{b_i}$.
			\item $\psi(\widehat{\delta})=1$ in cases~(\ref{it:nonorientodd}) and~(\ref{it:nonorienteven}).
			\item $\psi(v)=g^{w}$, where $0\leq w\leq n-1$ is such that $\theta_\tau(v)=(1,n,\ldots,2)^{w}$ in case~(\ref{it:nonorienteven}).
		\end{enumerate}
		Noting that $\theta_\tau(\widehat{\delta})$ is trivial by hypothesis and $\pi_2(g)=(1,n,\ldots,2)$, one may check in each of the three cases~(\ref{it:orient})--(\ref{it:nonorienteven}) that $\psi$ defines a homomorphism for which the diagram~(\ref{diag:borsuk_r2}) is commutative. 
	\end{proof}
	
	Although the conclusion of Proposition~\ref{prop:2} is the same as that of Proposition~\ref{prop:1}, its proof is somewhat different in nature. We first give a geometric proof of Proposition~\ref{prop:2} in a specific case. This will then enable us to give an algebraic proof in the general case. We regard $\torus$ as the quotient space $\rtwo/\z^2$. Notationally, we shall not distinguish between an element of $\rtwo$ and its image under the canonical projection in $\torus$. Let $x$ and $y$ denote the pointed homotopy classes of the loops in $\torus$ given by $(t,0)_{t\in [0,1]}$ and $(0,t)_{t\in [0,1]}$ respectively. Then $\pi_1(\torus)=\left\langle x, y \, | \, xyx^{-1}y^{-1} \right\rangle$.
	
	\begin{lemma}\label{lem:geometry}
		Given $k \in \mathbb{N}$, let $\tau \colon\thinspace \torus \to \torus$ be the map defined by $\tau(a,b)=(-a,a+b+\frac{1}{4k})$ for all $(a,b)\in \torus$. Then $\tau$ defines a free action of $\z_{4k}$ on $\torus$ such that the orbit space is the Klein bottle $\klein$, and for which induced short exact sequence~(\ref{seq:main}) associated to the covering $p_\tau \colon\thinspace \torus \to \klein$ is of the form:
		\begin{equation}\label{seq:geometry}\xymatrix{
				1 \ar[r] &
				\pi_1 (\torus)=\left\langle x, y \, \mid \, xyx^{-1}y^{-1} =1 \right\rangle \ar[r]^-{(p_\tau)_\#} &
				\pi_1(\klein)=\left\langle u, v \, \mid \, uvuv^{-1} =1 \right\rangle \ar[r]^-{\theta_\tau}  & \mathbb{Z}_{4k} \ar[r] & 1, 	
			}\; \end{equation}
		where:
		\begin{equation}\label{seq:geometry_roles}
			\text{$(p_\tau)_\#\colon\thinspace \begin{cases}
					x \longmapsto uv^{2k} \\
					y \longmapsto v^{4k}
				\end{cases}$ 
				and \quad
				$\theta_\tau\colon\thinspace \begin{cases}
					u \longmapsto \overline{2k} \\
					v \longmapsto \overline{1}.
				\end{cases}$}
		\end{equation}
		Moreover, $(\torus,\mathbb{Z}_{4k},\tau;\rtwo)$ does not have the Borsuk-Ulam property.
	\end{lemma}
	
	\begin{proof}
		Let $k\in \mathbb{N}$, and let $\tau \colon\thinspace \torus \to \torus$ be the map defined by $\tau(a,b)=(-a,a+b+\frac{1}{4k})$ for all $(a,b)\in \torus$. For each  $0 \leq i \leq 4k$, we have: 
		\begin{equation}\label{eq:tau_i}
			\tau^i(a,b)=\begin{cases}
				(a,b+\frac{i}{4k}) &\text{if $i$ is even} \\
				(-a,a+b+\frac{i}{4k}) &\text{if $i$ is odd}.
			\end{cases}
		\end{equation}
		One may check that if $1 \leq i < j \leq 4k$, then $\tau^i(a,b) \neq \tau^j(a,b)$ (note that if $a=1/2$ and $a+\frac{j-i}{4k}\in \z$ then $i$ and $j$ have the same parity), and since $\tau^{4k}=\operatorname{\text{Id}}_{\torus}$, it follows that $\tau$ defines a free action of $\z_{4k}$ on $\torus$. Let $g\colon\thinspace \torus \to \klein$ be the double covering map whose non-trivial associated deck transformation $\tau_1 \colon\thinspace \torus \to \torus$ is given by $\tau_1(a,b)=(-a,b+\frac{1}{2})$, so $g\circ \tau_1=g$, and let $u$ and $v$ be the homotopy classes determined by the loops $(g(t,0))_{t \in [0,1]}$ and $(g(0,t/2))_{t \in [0,1]}$ respectively. Then $\pi_1(\klein)=\left\langle u, v \, | \, uvuv^{-1}=1 \right\rangle$, and the induced homomorphism  $g_\# \colon\thinspace \pi_1(\torus) \to \pi_1(\klein)$ is given by $g_\#(x)=u$ and $g_\#(y)=v^2$.
		%
		%
		Let $f \colon\thinspace \torus \to \torus$ be defined by $f(a,b)=(a,ka+2kb)$ for all $(a,b)\in \torus$. The induced homomorphism $f_\# \colon\thinspace \pi_1(\torus) \to \pi_1(\torus)$ is given by $f_\#(x)=xy^k$ and $f_\#(y)=y^{2k}$. 
		%
		%
		Notice that $f$ is a $2k$-sheeted covering of $\torus$, and that $f_\#$ is injective. By construction, the map $g \circ f \colon\thinspace \torus \to \klein$ is a $4k$-sheeted covering of $\klein$ by $\torus$, and consequently the deck transformation group associated to the covering $g \circ f$ is of order $4k$. We claim that $\tau$ is a deck transformation for the covering map $g \circ f$. If this is the case, since $\tau$ defines a free action of $\z_{4k}$ on $\torus$, it follows that group of deck transformations for the covering map $g \circ f$ is isomorphic to $\z_{4k}$ and is generated by $\tau$, thus the associated orbit space $\torus_\tau$ is homeomorphic to $\klein$, and the corresponding quotient map $p_\tau$ is equal, up to homeomorphism, to $g \circ f$. To prove the claim, for all $(a,b) \in \torus$, we have:
		\begin{align*}
			(g \circ f)(\tau(a,b)) &= (g\circ f)\left(-a,a+b+\frac{1}{4k}\right)= g\left(-a,-ka+2k\left(a+b+\frac{1}{4k}\right)\right)\\
			&=g\left(-a,ka+2kb + \frac{1}{2}\right)=(g \circ \tau_1)\left(a,ka+2kb\right)= g(a,ka+2kb)=(g \circ f)(a,b).
		\end{align*}
		This proves the claim. Thus $(p_\tau)_\# \colon\thinspace \pi_1 (\torus) \to \pi_1 ( \klein)$ is injective, and satisfies the first equality of~(\ref{seq:geometry_roles}).
		Let $\theta_\tau\colon\thinspace \pi_1(\klein) \to \z_{4k}$ be the map defined by~(\ref{seq:geometry_roles}). One may check that $\theta_\tau$ is a well-defined surjective homomorphism. Let us show that the sequence~(\ref{seq:geometry}) is short exact. To do so, it remains to show that $\operatorname{\text{Im}}( (p_\tau)_\#)= \ker{\theta_\tau}$. First let $z\in \operatorname{\text{Im}}((p_\tau)_\#)$. Then there exist $m,n\in \z$ such that $z=(p_\tau)_\#(x^m y^n)=(uv^{2k})^m (v^{4k})^n$, and thus $\theta_\tau(z)= m(\overline{2k}+\overline{2k})+n\ldotp  \overline{4k}=\overline{0}$, hence $\operatorname{\text{Im}}( (p_\tau)_\#) \subset \ker{\theta_\tau}$. Conversely, let $y\in \ker{\theta_\tau}$. Now $y\in \pi_1(\klein)$, so there exist $m,n\in \z$ such that $y=u^m v^n$, and $\overline{0}=\theta_\tau(y)=m \ldotp \overline{2k}+ n\ldotp \overline{1}$. Thus $2mk+n \equiv 0 \bmod{4k}$, and there exists $l \in \z$ such that $n=2mk + 4lk$.
		%
		%
		%
		%
		Since $vuv^{-1}=u^{-1}$, we see that $v^2uv^{-2}=u$, so $u$ and $v^2$ commute, and thus:
		\begin{equation*}
			y=u^m v^n=u^m (v^{2k})^m (v^{4k})^l=(uv^{2k})^m (v^{4k})^l=((p_\tau)_\#(x))^m ((p_\tau)_\#(y))^l=(p_\tau)_\#(x^m y^l) \in \operatorname{\text{Im}}( (p_\tau)_\#).
		\end{equation*}
		Therefore $\ker{\theta_\tau} \subset \operatorname{\text{Im}}( (p_\tau)_\#)$. Hence $\ker{\theta_\tau}= \operatorname{\text{Im}}( (p_\tau)_\#)$, and we conclude that the sequence~(\ref{seq:geometry}) is short exact.
		We conclude that the short exact sequence associated to the covering $p_\tau \colon\thinspace \torus \to \klein$ given by~(\ref{seq:main}) is~(\ref{seq:geometry}). 
		
		
		It remains to show that $(\torus,\z_{4k},\tau;\rtwo)$ does not have the Borsuk-Ulam property. It suffices to exhibit an explicit map $h\colon\thinspace \torus \to \rtwo$ such that $h(\tau^i(a,b)) \neq h(\tau^j(a,b))$ for each $(a,b)\in \torus$ and $0 \leq i < j \leq 4k-1$. First note that by~(\ref{eq:tau_i}), the orbit of a point $(a,b) \in \torus$ by the action of $\z_{4k}$ generated by $\tau$ belongs to the union of the vertical circles $((a,t))_{t\in [0,1]}$ and $((1-a,t))_{t\in [0,1]}$, and these two circles coincide if $a\in \{ 0, \frac{1}{2}\}$. Let $C\subset \rtwo$ be the annulus defined with respect to polar coordinates by $C=\{ (r,\theta) \, \mid \, \text{$1\leq r\leq 2$ and $\theta \in [0,2\pi)$} \}$, and let $h \colon\thinspace \torus \to \rtwo$ be defined by $h(a,b)=(r(a), 2\pi b)$ for all $(a,b)\in \torus$ (we take $0\leq a,b\leq 1$), where:
		\begin{equation*}
			r(a)=\begin{cases}
				(3-4a)/2 & \text{if $0 \leq a \leq \frac{1}{4}$}\\
				(1+4a)/2 & \text{if $\frac{1}{4} \leq a \leq \frac{3}{4}$}\\
				(7-4a)/2 & \text{if $\frac{3}{4} \leq a \leq 1$.}
			\end{cases}
		\end{equation*}
		One may check that $h$ is well defined and continuous. Further, its restriction to each circle $((a,t))_{t\in [0,1]}$ of $\torus$, where $0\leq a \leq 1$, is injective, and that if $a\in (0,1) \setminus \{ \frac{1}{2}\}$ then the two circles $(h(a,t))_{t\in [0,1]}$ and $(h(1-a,t))_{t\in [0,1]}$ are disjoint. 
		We conclude that the restriction of the map $h$ to each orbit of the action $\tau$ of $\z_{4k}$ is injective, and this completes the proof of the lemma.
		%
		%
	\end{proof}

	\begin{lemma}\label{lem:algebra} 
		Given $k \in \mathbb{N}$, there exist $\alpha,\beta \in B_{\z_{4k}}(\rtwo)$ such that $\pi_2(\alpha)=\overline{2k}$, $\pi_2(\beta)=\overline{1}$ and $\alpha\beta\alpha\beta^{-1}=1$ in $B_{\z_{4k}}(\rtwo)$.
	\end{lemma}
	
	\begin{proof}
Let $\theta_\tau \colon\thinspace \pi_1(\klein) \to \z_{4k}$ be the homomorphism defined in~(\ref{seq:geometry_roles}). By Corollary~\ref{cor:borsuk_braid_r2} and Lemma~\ref{lem:geometry}, there exists a homomorphism $\psi \colon\thinspace \pi_1(\klein) \to B_{\z_{4k}}(\rtwo)$ such that the following diagram is commutative:
		$$\xymatrix{
			\left\langle u,v \, | \, uvuv^{-1}=1 \right\rangle \ar@{.>}[rr]^-{\psi} \ar[rd]_-{\theta_\tau}  & & B_{\z_{4k}}(\rtwo) \ar[ld]^-{\pi_2}  \\
			&   \z_{4k}. &
		}$$
		%
		%
		It then suffices to take $\alpha=\psi(u)$ and $\beta=\psi(v)$.
	\end{proof}
	
	
	\begin{proof}[Proof of Proposition~\ref{prop:2}]
		We will define a homomorphism $\psi\colon\thinspace \pi_1(M_\tau) \to B_{\z_{4k}}(\rtwo)$ for which the diagram~(\ref{diag:borsuk_r2}) is commutative, from which it will follow from Corollary~\ref{cor:borsuk_braid_r2} that $(M,\z_{4k},\tau;\rtwo)$ does not have the Borsuk-Ulam property. For this, we make use of the elements $\alpha$ and $\beta$ given in Lemma~\ref{lem:algebra} and the element $g$ given in Remark~\ref{rems:gensbzn}(\ref{it:gensbzna}). 
		We endow $\pi_1(M_\tau)$ with one of the presentations of cases~(\ref{it:nonorientodd}) and~(\ref{it:nonorienteven}) given at the beginning of this section. To simplify the analysis, we write $\pi_1(M_\tau)=\left\langle \left. \widehat{\delta},v,a_1,a_2,\ldots,a_{2m-1},a_{2m} \, \right\rvert \, \widehat{\delta} v \widehat{\delta} v^{-1}[a_1,a_2]\cdots [a_{2m-1},a_{2m}]=1 \right\rangle$, where we take $v=1$ in case~(\ref{it:nonorientodd}). By hypothesis, $\theta_{\tau}(\widehat{\delta})= \overline{2k}$. Let $0\leq z \leq 4k-1$ be such that $\theta_\tau(v)=\overline{z}$ (so $z=0$ if $v=1$). For $i=1,\ldots, 2m$, let $b_i\in \z$ be such that $\theta_{\tau}(a_i)=\overline{b_i}$, and let $\psi(v)=\beta^{z}$. If $z$ is odd, we define $\psi$ on the remaining generators of $\pi_1(M_\tau)$ as follows:
		\begin{equation}\label{eq:psiodd}
			\psi\colon\thinspace \begin{cases}
				\widehat{\delta} \longmapsto \alpha \\
				a_i \longmapsto g^{b_i}, \;\text{where $1 \leq i \leq 2m$.}
			\end{cases}
		\end{equation}
		If $z$ is even, then by the surjectivity of $\theta_{\tau}$ there exist $1\leq j\leq 2m$ and $l\in \z$ odd such that $\theta_\tau(a_j)=\overline{l}$. If $j$ is odd (resp.\ even), let $\widehat{\jmath}=j+1$ (resp.\ $\widehat{\jmath}=j-1$). Then we define $\psi$ on the remaining generators of $\pi_1(M_\tau)$ as follows:
		\begin{equation}\label{eq:psieven}
			\psi\colon\thinspace \begin{cases}
				\widehat{\delta} \longmapsto \alpha^{(-1)^j} \\
				a_j \longmapsto \beta^{l}   \\
				a_{\widehat{\jmath}} \longmapsto \alpha^{-1} \beta^{b_{\widehat{\jmath}}+2k} \\
				a_i \longmapsto g^{b_i}\; \text{if $i \neq j,\widehat{\jmath}$.}
			\end{cases}
		\end{equation}
		Let us show that $\psi$ extends to a homomorphism from $\pi_1(M_\tau)$ to $B_{\z_{4k}}(\rtwo)$. If $z$ is odd, then we are in case~(\ref{it:nonorientodd}), and by~(\ref{eq:psiodd}), we have:
		\begin{equation*}
			\psi(\widehat{\delta} v \widehat{\delta} v^{-1}[a_1,a_2]\cdots [a_{2m-1},a_{2m}])= \psi(uvuv^{-1})= \alpha \ldotp \beta^z \alpha \beta^{-1}=\alpha \ldotp \alpha^{-1}=1,
		\end{equation*}
		where we have used Lemma~\ref{lem:algebra} and the fact that $z$ is odd. If $z$ is even, then by~(\ref{eq:psieven}), we have:
		\begin{align*}
			\psi(\widehat{\delta} v \widehat{\delta} v^{-1}[a_1,a_2]\cdots [a_{2m-1},a_{2m}])&= \psi(\widehat{\delta} v \widehat{\delta} v^{-1}[a_j,a_{\jmath}]^{(-1)^{j+1}})\\
			&= \alpha^{(-1)^j} \ldotp \beta^z \alpha^{(-1)^j} \beta^{-z} (\beta^l \alpha^{-1} \beta^{b_{\widehat{\jmath}}+2k} \beta^{-l} \ldotp \beta^{-b_{\widehat{\jmath}}-2k} \alpha)^{(-1)^{j+1}}\\
			&=\alpha^{2(-1)^j} \ldotp \alpha^{(-1)^{j+1}}=1,
		\end{align*}
		where we have used Lemma~\ref{lem:algebra} and the fact that $z$ is even and $l$ is odd. This proves that $\psi$ is a homomorphism. One may 
		check that $\pi_2\circ \psi=\theta_{\tau}$, and hence the diagram~(\ref{diag:borsuk_r2}) is commutative. The result then follows from Corollary~\ref{cor:borsuk_braid_r2}.
	\end{proof}
	
	\begin{proof}[Proof of Proposition~\ref{prop:3}]
		Suppose on the contrary that $(M,\z_{4k+2},\tau;\rtwo)$ does not have the Borsuk-Ulam property. Then by Corollary~\ref{cor:borsuk_braid_r2} there exists a homomorphism $\psi \colon\thinspace \pi_1(M_\tau) \to B_{\z_{4k+2}}(\rtwo)$ such that the diagram~(\ref{diag:borsuk_r2}) is commutative. Suppose that $M_{\tau}$ is non-orientable of odd (resp.\ even) genus, so that $\pi_1(M_{\tau})$ is as in case~(\ref{it:nonorientodd}) (resp.\ case~(\ref{it:nonorienteven})) given at the beginning of this section. By hypothesis, $\theta_{\tau}(\widehat{\delta})= \overline{2k+1}$, where $\widehat{\delta}=c$ (resp.\ $\widehat{\delta}=u$). From Theorem~\ref{th:bzn_rtwo}, for $i=1,\ldots,2m$, there exist $w_i\in P_{4k+2}(\rtwo)$ and $0\leq b_i<n$, $w_{\widehat{\delta}} \in P_{4k+2}(\rtwo)$ (resp.\ $w_v \in P_{4k+2}(\rtwo)$ and $0\leq b<n$) such that for $i=1,\ldots,2m$, $\theta_\tau(a_i)=\overline{b_i}$ and $\psi(a_i)=w_i g^{b_i}$,  $\psi(\widehat{\delta})=w_{\widehat{\delta}} g^{2k+1}$, and if $M_{\tau}$ is of even genus, $\theta_\tau(v)=\overline{b}$ and $\psi(v)=w_v g^{b}$.
		%
		%
		%
		The evaluation map $\varepsilon \colon\thinspace P_{4k+2}(\rtwo) \to \z$ defined on the generating set $(A_{i,j})_{1\leq i<j\leq n}$ of $P_{4k+2}(\rtwo)$ by $\varepsilon(A_{i,j})=1$ for all $1\leq i<j\leq n$ extends to a well-defined homomorphism that satisfies $\varepsilon(gwg^{-1})= \varepsilon(w)$ for all $w \in P_{4k+2}(\rtwo)$. For all $i=1,\ldots,m$, it follows that:
		\begin{equation}\label{eq:aux_2}
			\varepsilon \circ \psi([a_{2i-1},a_{2i}])= \varepsilon\bigl(w_{2i-1} \ldotp g^{b_{2i-1}} w_{2i} g^{-b_{2i-1}} \ldotp g^{b_{2i}} w_{2i-1}^{-1} g^{-b_{2i}} \ldotp w_{2i}^{-1}\bigr)=\varepsilon([w_{2i-1},w_{2i}])=0.
		\end{equation}
		Applying~(\ref{eq:aux_2}) and relation~(\ref{rel:IV}) of Theorem~\ref{th:bzn_rtwo}, if $M_{\tau}$ is of odd genus, we obtain:
		\begin{equation*}
			0 =\varepsilon \circ \psi\left(  c^2[a_1,a_2]\cdots [a_{2m-1},a_{2m}] \right)=\varepsilon \circ \psi(c^2)= \varepsilon(w_c \ldotp g^{2k+1} w_c g^{-2k-1} \ldotp \Delta_{4k+2}^2)= 2\varepsilon(w_c)+\varepsilon (\Delta_{4k+2}^2),
		\end{equation*}
		while if $M_{\tau}$ is of even genus, we see that:
		\begin{align*}
			0&= \varepsilon \circ \psi(uvuv^{-1}[a_1,a_2]\cdots [a_{2m-1},a_{2m}])=\varepsilon \circ \psi(uvuv^{-1})\\
			&= \varepsilon \bigl( w_u \ldotp g^{2k+1} w_v g^{-2k-1} \ldotp g^{2k+1+b} w_u g^{-2k-1-b} \ldotp
			\Delta_{4k+2}^2 \ldotp w_v^{-1}\bigr)= 2\varepsilon(w_u)+\varepsilon (\Delta_{4k+2}^2).
		\end{align*}
		Combining the two cases of odd and even genus, we see that: 
		\begin{equation}\label{eq:vareps}
			2\varepsilon(w_{\widehat{\delta}})+\varepsilon(\Delta_{4k+2}^2)=0.
		\end{equation}
		It follows from relation~(\ref{rel:IV}) of Theorem~\ref{th:bzn_rtwo} that $\varepsilon(\Delta_{4k+2}^2)=\varepsilon\left( \prod_{j=2}^{4k+2}\left( \prod_{i=1}^{j-1} A_{i,j} \right)\right)= (2k+1)(4k+1)$. Taking equation~(\ref{eq:vareps}) modulo $2$, we obtain a contradiction, and hence $(M,\z_{4k+2},\tau;\rtwo)$ has the Borsuk-Ulam property.
	\end{proof}


	

\section{Two cases where $G$ is the symmetric group}\label{sec:exam_sigma1}

	
	
	Let $N$ be a compact surface without boundary, and let $G$ be a non-trivial finite group for which there exists a surjective homomorphism $\theta\colon\thinspace \pi_1(N) \to G$. Let $\iota\colon\thinspace \ker{\theta} \to \pi_1(N)$ denote the inclusion. By standard covering space arguments, there exists a compact surface without boundary $M$ and a free action $\tau\colon\thinspace G \times M \to M$ such that the associated exact sequence given by~(\ref{seq:main}) is of the following form:
	\begin{equation}\label{eq:provenient1}
		\xymatrix{
			1 \ar[rr] && \pi_1 (M)=\ker{\theta} \ar[rr]^-{(p_\tau)_\#=\iota} && \pi_1(M_\tau)=\pi_1(N) \ar[rr]^-{\theta_\tau=\theta} && G \ar[rr] && 1. \\
	}\end{equation}
	We say that the surface $M$ and the action $\tau$ arise from $\theta$. This construction provides us with a method to construct a compact surface without boundary on which $G$ acts freely. We will consider examples where $G$ is either $\Sigma_n$ or $\z_n$, where $n>2$.
	
Let $N_1=\mathbb{T}^2 \mathop{\#} \mathbb{T}^2$ be the connected sum 
 of two copies of $\torus$, let $N_2=\mathbb{RP}^2 \mathop{\#} \mathbb{T}^2$ be the connected sum
  of $\mathbb{RP}^2$ and $\torus$, and let $\theta_1\colon\thinspace \pi_1(N_1)=\left\langle a_1,b_1,  a_2, b_2 \, | \, [a_1,b_1][a_2,b_2]=1 \right\rangle \to \Sigma_n$ and $\theta_2\colon\thinspace \pi_1(N_2)=\left\langle c, a_1,b_1  \, | \, c^2[a_1,b_1]=1 \right\rangle \to \Sigma_n$ be given by: 
	\begin{align}
		\theta_1 (a_1)& = \text{$\theta_1 (b_1)=(1,2)$, and $\theta_1 (a_2)=\theta_1 (b_2)=(1, n, \ldots, 2)$}\label{eq:n1}\\
		\theta_2 (c)&= \text{$(1,2)$, and $\theta_2 (a_1)=\theta_2 (b_1)=(1, n, \ldots, 2)$}.\label{eq:n2}
	\end{align}
	Note that $\theta_1$ and $\theta_2$ define surjective homomorphisms. For $i=1,2$, let $M_i$ and $\tau_i\colon\thinspace \Sigma_n \times M_i \to M_i$ be the compact surface without boundary and the free action arising from $\theta_i$ respectively.
	
	%
	%
	
	%
	%
	
	The following result suggests that there are many possibilities to be explored in the generalisation of the Borsuk-Ulam theorem to actions of finite groups on surfaces.

	\begin{proposition}\label{prop:ex121}
		Let $n>2$. With the above notation, we have the following:
		\begin{enumerate}
			\item\label{it:ex12a1} the quadruple $(M_1, \Sigma_n, \tau_1; \mathbb{R}^2)$ does not have the Borsuk-Ulam property.
			\item\label{it:ex12b1} for the quadruple $(M_2, \Sigma_n, \tau_2; \mathbb{R}^2)$, there does not exist a homomorphism $\psi$ for which the lower triangle of the diagram~(\ref{diag:diag_equiv}) is commutative.  
			\item\label{it:ex12c1} the quadruple $(M_2, \z_n, \tau_{\z_n}; \mathbb{R}^2)$, where the action of $\z_n$ is given by the restriction of the action of $\Sigma_n$ corresponding to the action of the cyclic subgroup $\langle (1,n\ldots,2)\rangle$, does not have the Borsuk-Ulam property.
			
		\end{enumerate}
	\end{proposition}
	
	
	
	
\begin{remark}
If $n=2$, a stronger version of Proposition~\ref{prop:ex121} may be found in~\cite{Gon}. Note that the quadruple $(M_1, \Sigma_2 = \z_2, \tau_1; \mathbb{R}^2)$ does not have the Borsuk-Ulam property, while the quadruple $(M_2, \Sigma_2 = \z_2, \tau_2; \mathbb{R}^2)$ does have the Borsuk-Ulam property. 
\end{remark}
	
	\begin{proof}[Proof of Proposition~\ref{prop:ex121}]
		By construction, first note that for $i=1,2$, $\pi_1(M_{\tau_i})=\pi_1(N_i)$.
		\begin{enumerate}
			\item Let $\psi\colon\thinspace \pi_1(M_{\tau_1}) \to B_n(\rtwo)$ be the homomorphism defined on the generators of $\pi_1(M_{\tau_1})$ by $\psi (a_1)=\psi (b_1)=\sigma_1$ and $\psi (a_2)=\psi (b_2)=g$,
			%
			%
			where $g$ is as defined in Remarks~\ref{rems:gensbzn}(\ref{it:gensbzna}). Using~(\ref{eq:provenient1}) and~(\ref{eq:n1}), we may check that 
			$\pi\circ \psi=\theta_{\tau_1}$, and Corollary~\ref{cor:borsuk_braid_r2} implies that $(M_1, \Sigma_n, \tau_1; \mathbb{R}^2)$ does not have the Borsuk-Ulam property.
			
			\item Suppose on the contrary that there exists a homomorphism $\psi$ for which the lower triangle of the diagram~(\ref{diag:diag_equiv}) is commutative. From the lower short exact sequence of~(\ref{eq:diag_braids}) and~(\ref{eq:n2}), there exist $x,y,z \in P_n(\rtwo)$ such that $\psi(c)=x \sigma_1$, $\psi(a_1)=yg$ and $\psi(b_1)=zg$,  
			%
			%
			where $g$ is as defined in Remarks~\ref{rems:gensbzn}(\ref{it:gensbzna}). Taking the image of the relation of $\pi_1(M_{\tau_2})$ and using~(\ref{eq:defaij}), we obtain:
			\begin{equation}\label{eq:aux41}
				1= \psi \left( c^2 [a_1,b_1] \right)=x \sigma_1 x \sigma_1 y g z  y^{-1} g^{-1} z^{-1}=x \ldotp \sigma_1 x \sigma_1^{-1} \ldotp A_{1,2} \ldotp y  \ldotp g z g^{-1}\ldotp g y^{-1} g^{-1}\ldotp z^{-1}.
			\end{equation}
			Note that $\sigma_1 x \sigma_1^{-1}$, $g z g^{-1}$ and $g y^{-1} g^{-1}$ belong to $P_n(\rtwo)$ because $P_n(\rtwo)$ is a normal subgroup of $B_n(\rtwo)$. Let $\varepsilon\colon\thinspace P_{n}(\rtwo) \to \z$ be the evaluation homomorphism defined in the proof of Proposition~\ref{prop:3}.
			By~\cite[Lemma~3.1]{LiWu} and Theorem~\ref{th:bzn_rtwo}, we have  $\varepsilon(\sigma_1 A_{i,j} \sigma_1^{-1})=\varepsilon(g A_{i,j} g^{-1}) =\varepsilon(A_{i,j})=1$ for all $1 \leq i < j \leq n$. So applying $\varepsilon$ to~(\ref{eq:aux41}), we obtain $0=2 \varepsilon(x) + 1$ in $\z$, which yields a contradiction.
			
			\item If $n\not\equiv 2 \bmod{4}$, the result is clear by Theorem~\ref{th:main}. So let us assume that $n\equiv 2 \bmod{4}$, so $n = 4k+2$ for some $k \in \z$. 
We claim that the covering space $M_2$ is non-orientable, therefore $M_{\tau_{\z_n}}$ is non-orientable. To see that $M_2$ is non-orientable, for $i=1,\ldots, 2k+1$, the transposition $(1,2)$ is conjugate in $\Sigma_{4k+2}$ to the transposition $(i,2k+1+i)$, so there exists $\lambda_i\in \Sigma_{4k+2}$ such that $\lambda_i(1,2)\lambda_i^{-1}=(i,2k+1+i)$.
			For $i=1,\ldots, 2k+1$, let $\widetilde{\lambda_i}\in \pi_1(N_2)$ be such that $\theta_2(\widetilde{\lambda_i})= \lambda_i$.
Now let $w=\left( \prod_{i=1}^{2k+1} \lambda_i c \lambda_i^{-1} \right) a_1^{2k+1}$. Since $c$ represents a non-orientable loop in $N_2$, $2k+1$ is odd and $a_1$ is orientable, it follows that $w$ represents a non-orientable loop in $N_2$. 
			One may check that $w\in \ker{\theta_2}$, from which it follows that $M_2$ is non-orientable, and thus $(M_2)_{\tau_{\z_n}}$ is also non-orientable.
We claim that the homomorphism $(\theta_{\tau_{\z_n}})_\ab \colon\thinspace H_1( (M_2)_{\tau_{\z_n}},\z ) \to \z_{4k+2}$ induced by $\theta_{2}$ satisfies $(\theta_{\tau_{\z_n}})_\ab(\delta)=\overline{0}$, where $\delta$ is the unique non-trivial torsion element of $H_1( (M_2)_{\tau_{\z_n}})$. This being the case, it follows from Theorem~\ref{th:main} that the Borsuk-Ulam property does not hold for the quadruple $(M_2, \z_n, \tau_{\z_n}; \mathbb{R}^2)$. To prove the claim, if $p\colon\thinspace (M_2)_{\tau_{\z_n}} \to N_2$ is the covering map such that $p \circ p_{\tau_{\z_n}} = p_{\tau_2}$, it suffices to show that $\delta \in \ker{(p_\#)_{Ab}}$, where $(p_\#)_{Ab} \colon\thinspace \pi_1((M_2)_{\tau_{\z_n}})_{Ab} \to \pi_1(N_2)_{Ab}$ (or equivalently the homomorphism $p_\ast\colon\thinspace H_1((M_2)_{\tau_{\z_n}}) \to H_1(N_2)$ on the level of first homology) is the homomorphism induced by $p$. To see this, observe that the induced homomorphism $p^\ast\colon\thinspace H^2(N_2) \to H^2((M_2)_{\tau_{\z_n}})$ in cohomology with integral coefficients is trivial since $H^2(N_2)$ and $H^2((M_2)_{\tau_{\z_n}})$ are isomorphic to $\z_2$, and the homomorphism $p^{\ast}$ is multiplication by $(4k+1)!$, which is the cardinality of $\Sigma_n$ divided by $n$. Since $n>2$, the homomorphism $p^{\ast}$ is trivial. Using the universal coefficient formula, it follows that the torsion of $H_1( (M_2)_{\tau_{\z_n}} )$ is sent to  zero, and the result follows from  Theorem~\ref{th:main}.\qedhere
\end{enumerate}
\end{proof}

An immediate consequence of the Proposition~\ref{prop:ex121}(\ref{it:ex12a1}) is the following.

\begin{corollary}
	Given any finite group $G$, there exists a quadruple $(M,G, \tau;
	\mathbb{R}^2)$ for which the Borsuk-Ulam property does not hold.
\end{corollary}

Given a finite group $G$, it is not true in general that there exists
a quadruple $(M,G, \tau; \mathbb{R}^2)$ satisfying the Borsuk-Ulam
property. Theorem~\ref{th:main} provides many examples of this. For
example, if $G$ is cyclic of odd order, the quadruple $(M,G,\tau;
\mathbb{R}^2)$ does not satisfy the Borsuk-Ulam property. It is thus a
natural problem to classify the finite groups $G$ for which there
exists a quadruple $(M,G,\tau; \mathbb{R}^2)$ satisfying the Borsuk-Ulam
property.

	\subsection*{Acknowledgements}
	
	The work on this paper started during the Ph.D thesis~\cite{Laass} of the third author who was supported by the CNPq project n\textsuperscript{o} 140836 and the Capes/COFECUB project n\textsuperscript{o} 12693/13-8, and was completed during his Postdoctoral Internship at IME--USP from March 2020 to August 2021 that was supported by the Capes/INCTMat project n\textsuperscript{o} 88887.136371/2017-00-465591/2014-0. The first author is partially supported by the Projeto Tem\'atico FAPESP, grant n\textsuperscript{o} 2016/24707-4: {\it Topologia Alg\'ebrica, Geom\'etrica e Diferencial}.



\begin{thebibliography}{10}
		\labelsep=1em\relax
		
		\bibitem{Bir}{\sc J.~S.~Birman},
		{\em Braids, Links, and Mapping Class Groups},
		Annals of Mathematics Studies \textbf{82}, Princeton University Press (1974).
		
		\bibitem{Borsuk}{\sc K.~Borsuk},
		{\em Drei Sätze über die $n$-dimensionale Euklidische Sphäre},
		Fund.\ Math.\ \textbf{20} (1933), 177--190.
		
		\bibitem{FadNeu}{\sc  E.~Fadell, L.~Neuwirth},
		{\em Configuration Spaces},
		Math.\ Scand.\ \textbf{10} (1962), 111--118.
		
		
		\bibitem{Gon}{\sc D.~L.~Gon\c{c}alves},
		{\em The Borsuk-Ulam theorem for surfaces},
		Quaest.\ Math.\ \textbf{29} (2006), 117--123.
		
		\bibitem{GonGua}{\sc D.~L.~Gonçalves, J.~Guaschi},
		{\em The Borsuk-Ulam theorem for maps into a surface}, Top.\ Appl.\ \textbf{157} (2010), 1742--1759.
		
		\bibitem{GonGuaLaa1}{\sc D.~L.~Gon\c{c}alves, J.~Guaschi, V.~C.~Laass},
		{\em The Borsuk-Ulam property for homotopy classes of selfmaps of surfaces of Euler characteristic zero}, J.\ Fixed Point Theory Appl.\ (2019) 21:65.
		
		\bibitem{GonGuaLaa2}{\sc D.~L.~Gon\c{c}alves, J.~Guaschi, V.~C.~Laass},
		{\em The Borsuk-Ulam property for homotopy classes of maps from the torus to the Klein bottle},
		Topol.\ Methods Nonlinear Anal.\ \textbf{56} (2020), 529--558.
		
		\bibitem{GonGuaLaa3}{\sc D.~L.~Gon\c{c}alves, J.~Guaschi, V.~C.~Laass},
		{\em The Borsuk-Ulam property for homotopy classes of maps from the torus to the Klein bottle~--~part 2}, Topol.\ Methods Nonlinear Anal., to appear.
		
		
		\bibitem{GonSan}{\sc  D.~L.~Gon\c{c}alves, A.~P.~dos Santos},
		{\em Diagonal involutions and the Borsuk-Ulam property for product of surfaces}, Bull.\  Braz.\ Math.\ Soc.\ \textbf{50} no.~3  (2019),  771--786.
		
		
		
		
\bibitem{GonSanSil}{\sc  D.~L.~Gon\c{c}alves, A.~P.~dos Santos, W.~L.~Silva},  {\em The Borsuk--Ulam property for maps from product of two surfaces into a surface},  Topol.\ Methods Nonlinear Anal.\ \textbf{58} (2021), 367--388.
		
		
		
		
		\bibitem{Han}{\sc V.~L.~Hansen},
		{\em Braids and coverings: selected topics},
		London Mathematical Society Student Texts \textbf{18}, Cambridge University Press (1989).
		
		\bibitem{Johnson}{\sc D.~L.~Johnson}, 
		{\em Presentations of groups}, 
		London Mathematical Society Lecture Notes Series \textbf{22}, Cambridge University Press (1976).
		
		\bibitem{Laass}{\sc V.~C.~Laass},
		{\em A propriedade de Borsuk-Ulam para fun\c c\~oes entre superf\'icies}, 
		Ph.D Thesis, IME, Universidade de S\~ao Paulo (2015).
		
		\bibitem{LiWu}{\sc J.~Y.~Li, J.~Wu}, 
		{\em Artin braid groups and homotopy groups}, 
		Proc.\ London Math.\ Soc.\ \textbf{99} (2009), 521--556.
		
		
		\bibitem{Mura}{\sc K.~Murasugi, B.~I.~Kurpita},
		{\em A study of braids}, Mathematics and its Applications \textbf{484}, Kluwer Academic Publishers, (1999).
		
		\bibitem{Vick}{\sc J.~W.~Vick},
		{\em Homology Theory: An Introduction to Algebraic Topology}, 
		Graduate Texts in Mathematics \textbf{145}, Springer--Verlag (1994).
		
		\bibitem{White}{\sc G.~W.~Whitehead},
		{\em Elements of homotopy theory},
		Graduate Texts in Mathematics \textbf{61}, Springer--Verlag (1978).
		
		\bibitem{Zies}{\sc H.~Zieschang, E.~Vogt, H.-D.~Coldewey},
		{\em Surfaces and Planar Discontinuous Groups}, 
		Lecture Notes in Mathematics \textbf{835}, Springer--Verlag (1980).
		
		
	\end{thebibliography}
\end{document}